\documentclass[12pt]{article} 
\usepackage[english]{babel}
\usepackage[utf8]{inputenc} 
\usepackage{latexsym} 
\usepackage{amsfonts} 
\usepackage{amsmath}
\DeclareMathOperator*{\argsup}{arg\,sup}
\DeclareMathOperator*{\arginf}{arg\,inf}

\usepackage{amssymb}
\usepackage{MnSymbol}
\usepackage{caption}
\usepackage{subcaption}
\usepackage[colorlinks=true,urlcolor=blue]{hyperref}
\usepackage{listings}
\usepackage{tikz}

\usepackage{graphicx}
\usepackage{subcaption}

\usepackage{tabularx}
\usepackage{caption}
\usepackage{algorithm}
\usepackage{algorithmic}

\usepackage{amsthm}
\theoremstyle{definition} 
\newtheorem{theorem}{Theorem}[section] 
\newtheorem{corollary}[theorem]{Corollary}

\newtheorem{lemma}[theorem]{Lemma}

\newtheorem{definition}[theorem]{Definition}

\usepackage{color}
\definecolor{deepblue}{rgb}{0,0,0.5}
\definecolor{deepred}{rgb}{0.6,0,0}
\definecolor{deepgreen}{rgb}{0,0.5,0}

\definecolor{colKeys}{rgb}{0,0,1} 
\definecolor{colIdentifier}{rgb}{0,0,0} 
\definecolor{colComments}{rgb}{0,1,0.3} 
\definecolor{colString}{rgb}{0,0.5,0} 

\definecolor{dkgreen}{rgb}{0,0.6,0} 
\definecolor{gray}{rgb}{0.5,0.5,0.5} 

\graphicspath{ {./Figures/} }

\providecommand{\keywords}[1]{\textbf{\textit{Keywords:}} #1}
\providecommand{\subject}[1]{\textbf{\textit{Mathematics Subject Classification 2020:}} #1}

\title{Local Fourier Analysis of a Space-Time Multigrid Method for DG-SEM for the Linear Advection Equation} 

\author
{Lea M. Versbach$^{1\ast}$, Philipp Birken$^{1}$, Viktor Linders$^{1}$, Gregor Gassner$^{2}$\\
\\
\normalsize{$^{1}$Centre for Mathematical Sciences, Numerical Analysis, Lund University, Lund, Sweden}\\
\normalsize{$^{2}$Center for Data and Simulation Sciences; Department of Mathematics and Computer Science}\\
\normalsize{Weyertal 86-90, 50931 Köln, Germany}\\
\\
\normalsize{$^\ast$lea\_miko.versbach@math.lu.se}
}

\date{\today}

\begin{document}

\maketitle

\begin{abstract}
In this paper we present a Local Fourier Analysis of a space-time multigrid solver for a hyperbolic test problem. The space-time discretization is based on arbitrarily high order discontinuous Galerkin spectral element methods in time and a first order finite volume method in space. We apply a block Jacobi smoother and consider coarsening in space-time, as well as temporal coarsening only. Asymptotic convergence factors for the smoother and the two-grid method for both coarsening strategies are presented.
For high CFL numbers, the convergence factors for both strategies are $0.5$ for first order, and $0.375$ for second order accurate temporal approximations.
Numerical experiments in one and two spatial dimensions for space-time DG-SEM discretizations of varying order gives even better convergence rates of around $0.3$ and $0.25$ for sufficiently high CFL numbers.

\end{abstract}

\\
\subject{65M55,65M22,65M60,65T99}

\section{Introduction}

Space-time discontinuous Galerkin (DG) discretizations have received increased attention in recent years. One reason is that they allow for high order implicit discretizations and parallelization in time \cite{Gander2015}. Moreover, new space-time DG spectral element methods have been constructed \cite{Friedrich2019}. These are provably entropy stable for hyperbolic conservation laws which is of great interest in that community. Several authors have studied space-time DG methods for different equations, for instance hyperbolic problems in \cite{Dorfler2016,Friedrich2019}, advection-diffusion problems in \cite{Klaij2007,Sudirham2006}, the Euler equations of gas dynamics in \cite{VanderVegt2002,VanderVegt2002h} and nonlinear wave equations in \cite{VanderVegt2007}. 

The philosophy of space-time methods is to treat time as just an additional dimension \cite{Neumuller2013}. This has several advantages, i.e. moving boundaries can be treated more easily \cite{VanderVegt2006} and parallelization in time is possible \cite{Gander2015}. However, the technique also has challenges, since the temporal direction has a special role. Time always needs to follow a causality principle: a solution later in time is only determined by a solution earlier in time, never the other way around.

Several time parallel numerical methods exist, and can be divided into four groups \cite{Gander2015}: Methods based on multiple shooting, methods based on domain decomposition and waveform relaxation, space-time multigrid methods and direct time parallel methods. In this article we focus on space-time multigrid methods, which often scale linearly with the number of unknowns \cite{Trottenberg2001}. 

An analysis tool for multigrid methods is the Local Fourier Analysis (LFA), introduced in \cite{Brandt1977}. It can be used to study smoothers and two-grid algorithms. The technique is based on assuming periodic boundary conditions and transforming the given problem into the frequency domain using a discrete Fourier transform. Thus, the LFA can be used as a predictor for asymptotic convergence rates when considering problems with non-periodic boundary conditions \cite{Friedhoff2013}. The smoothing and asymptotic convergence are both related to the eigenvalues of the operators for the smoother and the two-grid algorithm. These operators are very large for space-time discretizations, since the effective dimension becomes $d+1$ for a $d$-dimensional problem. It is therefore not feasible to calculate the eigenvalues. When performing an LFA, the operators are of block diagonal form in the Fourier space, which reduces the problem to an analysis of so-called Fourier symbols. These are of much smaller size and make calculations feasible. Multigrid solvers have been analyzed in the DG context with block smoothers for convection-diffusion problems in \cite{Gopalakrishnan2003,Klaij2007,VanRaalte2005} and for elliptic problems in \cite{Hemker2003,Hemker2004}. Space-time MG methods have been analyzed mostly for parabolic problems \cite{Falgout2017,Friedhoff2013,Gander2016}. Analysis of space-time MG algorithms for DG discretizations of advection dominated flows has been quite limited but can be found for the advection-diffusion equation or linearized versions of the compressible Euler equations \cite{VanderVegt2012,VanderVegt2012h} and for generalized diffusion problems \cite{Friedhoff2013}.

In this article we use the LFA to analyze a space-time multigrid solver for a hyperbolic model problem. The analysis is similar to \cite{Gander2016}, where the authors considered a one-dimensional heat equation discretized with a finite element method in space and DG in time. Instead we study the one-dimensional linear advection equation discretized with a first order finite volume (FV) method in space and a discontinuous Galerkin spectral element method (DG-SEM) in time. This results in a fully discrete space-time DG discretization.

Due to the choice of the test problem as well as the spatial discretization we get complex Fourier symbols, making the analysis more difficult. However, for large CFL numbers we are able to determine asymptotic smoothing factors analytically. We use a block Jacobi smoother and compare two different coarsening strategies: coarsening in both temporal and spatial directions as well as coarsening in the temporal direction only. The two-grid convergence factors need to be calculated numerically based on the complex Fourier symbols. We compare the results of the analysis to numerical convergence rates for multi-dimensional advection problems with more general boundary conditions. These experiments are produced using the the Distributed and Unified Numerics Environment (DUNE), an open source modular toolbox for solving partial differential equations with grid-based methods as DG, finite element and finite differences \cite{Bastian2021,Dedner2021,Dedner2020,Dedner2010}.

In section \ref{SectionProblem} we describe the model problem and the DG space-time discretization. Moreover, we introduce the multigrid solver and its components. 
In section \ref{SectionBasicsLFA} we discuss some preliminaries: the equivalence of temporal DG-SEM discretizations with upwind flux and Lobatto IIIC Runge-Kutta methods, stability results, as well as the basic tools needed for the LFA. In section \ref{FourierSymbols} we derive the Fourier symbols for all elements of the multigrid iteration: discretization operator, smoother, restriction and prolongation. In section \ref{SectionSmoothingAnalysis} we analyze the smoothing properties for a block Jacobi smoother to find an optimal damping parameter. The two-grid asymptotic convergence factors are calculated in section \ref{SectionTwoGridAnalysis}. Numerical results for the advection equation in one and two dimensions are presented in section \ref{SectionNumericalResults} and serve as a comparison to the theoretical results obtained from the LFA. Conclusions are drawn in section \ref{SectionConclusions}.

\section{Problem Description}\label{SectionProblem}

The goal of this paper is to analyze a space-time multigrid solver for a discretized hyperbolic model problem. We consider the one-dimensional linear advection equation
\begin{align}\label{advectionProblem}
 u_t + a u_x = 0,~ (x,t)\in [L,R] \times [0,T] =: \boldsymbol{\Omega} \subset \mathbb{R}^2
\end{align}
with $a>0$. For the analysis, we need periodic boundary conditions in space and time.

\subsection{Discretization}

For the analysis we discretize \eqref{advectionProblem} with a space-time DG method with a FV method in space (which corresponds to one a DG method of order $p_x = 0$ with one Legendre-Gauss node, i.e., the midpoint) and variable temporal polynomial degree $p_t$. For the temporal direction, we use a DG-SEM. This is a specific DG discretization based on the following choices: The solution and the physical flux function are approximated element-wise by nodal polynomials using a Lagrangian basis based on Legendre-Gauss-Lobatto (LGL) nodes. Moreover, integrals are approximated by Gaussian quadrature, which is collocated with the nodes of the polynomial approximation \cite{Hesthaven2007}.

To discretize equation \eqref{advectionProblem}, we start by dividing $\boldsymbol{\Omega}$ into space-time elements $[x_n,x_{n+1}] \times [t_m,t_{m+1}],~n=1,\dots,N_x,~m=1,\dots,N$. A weak form on each space-time element $[x_n,x_{n+1}] \times [t_m,t_{m+1}]$ is
\begin{align*}
\int_{x_n}^{x_{n+1}} \int_{t_m}^{t_{m+1}} (u_t + au_x) \psi \text{d}x \text{d}t = 0,
\end{align*}
for test functions $\psi$ in a test space $C^1(\boldsymbol{\Omega})$.

It is of advantage for the DG-SEM discretization to map the temporal elements to the reference element $[-1,1]$ using the linear map $\tau(t) = 2 \frac{t-t_m}{\Delta t}$ with $\Delta t = t_{m+1}-t_m$. Then we get the modified weak form 
\begin{align*}
\frac{2}{\Delta t}\int_{x_n}^{x_{n+1}} \int_{-1}^{+1} u_{\tau}  \psi \text{d}x \text{d} \tau + \int_{x_n}^{x_{n+1}} \int_{-1}^{1} au_x \psi \text{d}x \text{d} \tau= 0.
\end{align*}

Integration by parts in both direction yields
\begin{align*}
\frac{2}{\Delta t}\int_{x_n}^{x_{n+1}} (u \psi|_{-1}^1  - \int_{-1}^{+1} u \psi_{\tau} \text{d} \tau) \text{d}x + a \int_{-1}^{1} (u \psi|_{x_n}^{x_{n+1}}  - \int_{x_n}^{x_{n+1}}u \psi_{\xi} \text{d}x) \text{d}\tau = 0.
\end{align*}

We now approximate $u$ on each space-time element $[x_n,x_{n+1}] \times [-1,1]$ by polynomials of variable degree $p_t$ in time and constant degree $p_x=0$ in space:
\begin{align*}
u(x,\tau) = \sum_{i=1}^{N_t} u_{ij} \ell_i(\tau),
\end{align*} 
with $N_t = p_t+1$ and $j=1,\dots,N_x$ volumes in space. For the DG-SEM discretization, the basis functions $\ell_i$ are Lagrange polynomials of degree $p_t$ based on the Legendre-Gauss-Lobatto (LGL) nodes $\{\tau_j \}_{i=1}^{N_t}$ in $[-1,1]$. We choose $\psi$ from the same space, thus
\begin{align*}
\psi(x,\tau) = \sum_{i=1}^{N_t} \psi_{ij} \ell_i(\tau),~ j=1,\dots,N_x.
\end{align*} 

Inserting the polynomial approximations gives
\begin{align*}
\frac{2}{\Delta t}\int_{x_n}^{x_{n+1}} (u \psi|_{-1}^1  - \int_{-1}^{+1} u \psi_{\tau} \text{d} \tau) \text{d}x + a \int_{-1}^{1} (u(x_{n+1},\tau)\psi(x_{n+1},\tau) - u(x_{n},\tau))\psi(x_{n},\tau) \text{d}\tau = 0.
\end{align*}

Next, we approximate the integrals in time with Gaussian quadrature using the same LGL nodes $\{\tau_i \}_{i=1}^{N_{\tau}}$ and weights $\{\omega_i \}_{i=1}^{N_t}$ as for the basis functions
\begin{align*}
\int_{-1}^{1} f(\tau) \text{d}\tau \approx \sum_{i=1}^{N_t} \omega_i f(\tau_i),
\end{align*}
and consider mean values in the spatial direction.
This yields
\begin{align*}
\frac{2}{\Delta t} (u^*_{N_tj}\delta_{iN_x} - u^*_{1j}\delta_{i1} - \sum_{l=1}^{N_t} \omega_l \ell_i'(\tau_l) u_{lj}) + \frac{a \omega_i}{\Delta x} (u^*_{i,j} - u^*_{i,j-1}) = 0,~ i=1,\dots,N_t, j=2,\dots,N_x,
\end{align*}
where we have replaced the boundary terms by numerical flux functions $u^*$. Here, we choose the upwind flux in the spatial and the temporal direction.

We can combine the temporal and spatial discretizations with a tensor product ansatz. Let us denote the spatial discretization operators by the index $\xi$ and the temporal discretization operators by $\tau$. With the temporal DG-SEM operators
\begin{align}\label{TemporalOperator}
& \mathbf{M}_{\tau} = \frac{\Delta t}{2} 
\begin{pmatrix}
\omega_1 \\
& \ddots \\
& & \omega_{N_t}
\end{pmatrix} \in \mathbb{R}^{N_t \times N_t},~
\mathbf{C}_{\tau} =
\begin{pmatrix}
0 & & 1 \\
& \ddots \\
& & 0
\end{pmatrix} \in \mathbb{R}^{N_t \times N_t},~
\mathbf{E}_{N_t} =
\begin{pmatrix}
0 \\
& \ddots \\
& & 0 \\
& & & 1
\end{pmatrix} \in \mathbb{R}^{N_t \times N_t}, \nonumber\\
& \mathbf{K}_{\tau} = \mathbf{E}_{N_t} - \mathbf{D}_{\tau}^T \mathbf{M}_{\tau} \in \mathbb{R}^{N_t \times N_t}, ~ (\mathbf{D}_{\tau})_{ij} = \frac{2}{\Delta t} \ell'_j(\tau_i),~ i,j = 1\dots,N_t, 
\end{align}
and the spatial operator
\begin{align}\label{SpatialOperator}
 \mathbf{K}_{\xi} = \frac{a}{\Delta x}
 \begin{pmatrix}
  1 & & & -1 \\
  -1 & 1 & & \\
  & \ddots & \ddots & \\
  & & -1 & 1
 \end{pmatrix}\in \mathbb{R}^{N_x \times N_x},~ \mathbf{I}_{\xi} \in \mathbb{R}^{N_x \times N_x},
\end{align}
the following linear space-time system has to be solved on each so-called space-time slab $n$:
\begin{align}\label{STadvection}
 (\mathbf{I}_{\xi}  \otimes \mathbf{K}_{\tau} + \mathbf{K}_{\xi}  \otimes \mathbf{M}_{\tau})\mathbf{u}^{n+1} = (\mathbf{I}_{\xi}  \otimes \mathbf{C}_{\tau}) \mathbf{u}^n,
\end{align}
Here, the spatial matrices correspond to the whole domain in space while the temporal matrices correspond to one DG element in time.

An example for an equidistant space-time grid can be seen in Figure~\ref{SpaceTimeGrid}. Then all $N_x$ spatial elements and one temporal element $n$ represent one space-time slab $n=1,\dots,N$, as highlighted gray in the figure. 

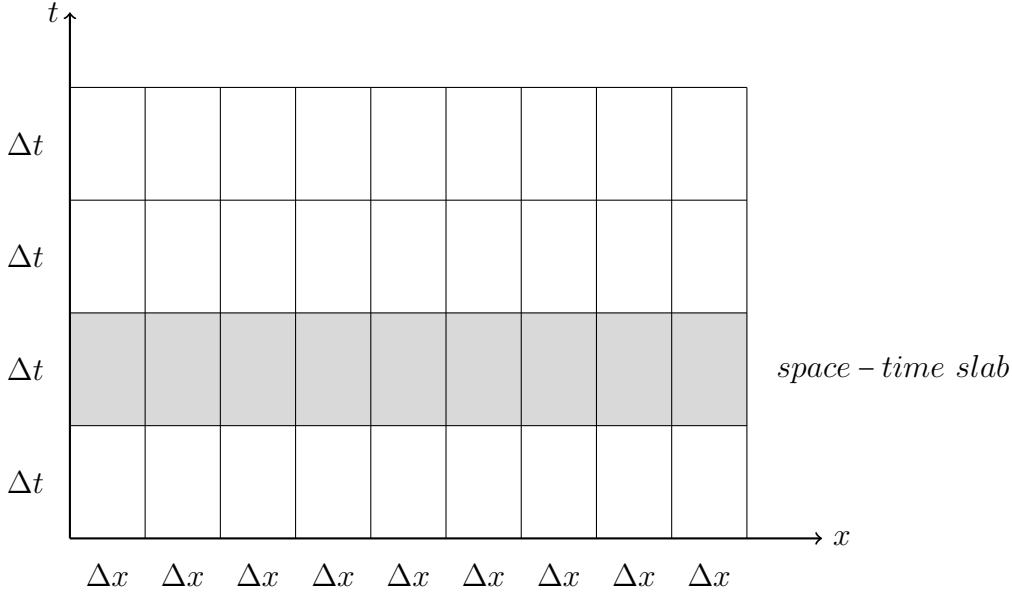
\begin{figure}
\centering
\begin{tikzpicture}
\draw[fill=gray!30,gray!30] (2,1.5) -- (2,3) -- (11,3) -- (11,1.5);
\node [right] at (11.25,2.25) {$space-time~slab$};

\draw [->,thick] (2,0) -- (12,0);
\node [right] at (12,0) {$x$};
\draw [->,thick] (2,0) -- (2,7);
\node [left] at (2,7) {$t$};

\draw [] (3,0) -- (3,6);
\draw [] (4,0) -- (4,6);
\draw [] (5,0) -- (5,6);
\draw [] (6,0) -- (6,6);
\draw [] (7,0) -- (7,6);
\draw [] (8,0) -- (8,6);
\draw [] (9,0) -- (9,6);
\draw [] (10,0) -- (10,6);
\draw [] (11,0) -- (11,6);

\draw [] (2,1.5) -- (11,1.5);
\draw [] (2,3) -- (11,3);
\draw [] (2,4.5) -- (11,4.5);
\draw [] (2,6) -- (11,6);

\node [left] at (1.8,0.75) {$\Delta t$};
\node [left] at (1.8,2.25) {$\Delta t$};
\node [left] at (1.8,3.75) {$\Delta t$};
\node [left] at (1.8,5.25) {$\Delta t$};

\node [below] at (2.5,-0.2) {$\Delta x$};
\node [below] at (3.5,-0.2) {$\Delta x$};
\node [below] at (4.5,-0.2) {$\Delta x$};
\node [below] at (5.5,-0.2) {$\Delta x$};
\node [below] at (6.5,-0.2) {$\Delta x$};
\node [below] at (7.5,-0.2) {$\Delta x$};
\node [below] at (8.5,-0.2) {$\Delta x$};
\node [below] at (9.5,-0.2) {$\Delta x$};
\node [below] at (10.5,-0.2) {$\Delta x$};

\end{tikzpicture}
\caption{Equidistant space-time grid in one spatial dimension.}\label{SpaceTimeGrid}
\end{figure}

Let us write the vector of unknowns for the space-time problem as $\underline{\mathbf{u}} = [\mathbf{u}^1,\dots,\mathbf{u}^N]^T \in \mathbb{R}^{N N_x N_t}$. The components of $\mathbf{u}^n \in \mathbb{R}^{N_x N_t}$ are given by $\mathbf{u}^n_{j,k} \in \mathbb{R}$, where $n$ denotes the space-time slab, $j$ is the index for the unknowns in space and $k$ the index for the unknowns in one time element, i.e. $\mathbf{u}^n_{j}\in \mathbb{R}^{N_t}$. This index notation is used throughout this paper for vectors in the space $\mathbb{R}^{N N_x N_t}$. 

The full space-time system for $N$ space-time slabs on $[0,T]$ can then be written in block form
\begin{align}\label{MG}
\underline{\mathbf{L}}_{\tau,\xi} \underline{\mathbf{u}} :=
 \begin{pmatrix}
  \mathbf{A}_{\tau,\xi} & & & \mathbf{B}_{\tau,\xi}\\
  \mathbf{B}_{\tau,\xi} & \mathbf{A}_{\tau,\xi} \\
  & \ddots & \ddots \\
  & & \mathbf{B}_{\tau,\xi} & \mathbf{A}_{\tau,\xi}
 \end{pmatrix}
 \begin{pmatrix}
  \mathbf{u}^1 \\
  \mathbf{u}^2 \\
  \vdots \\
  \mathbf{u}^N
 \end{pmatrix}
 =
 \begin{pmatrix}
  \mathbf{0} \\
  \mathbf{0} \\
  \vdots \\
  \mathbf{0}
 \end{pmatrix}
 =: \underline{\mathbf{b}}
\end{align}
with 
\begin{align}
\mathbf{A}_{\tau,\xi} & := \mathbf{I}_{\xi} \otimes \mathbf{K}_{\tau} + \mathbf{K}_{\xi} \otimes \mathbf{M}_{\tau} \in \mathbb{R}^{N_x N_t \times N_x N_t},\\
\mathbf{B}_{\tau,\xi} & := - \mathbf{I}_{\xi} \otimes \mathbf{C}_{\tau} \in \mathbb{R}^{N_x N_t \times N_x N_t}.
\end{align}
and the space-time operator $\underline{\mathbf{L}}_{\tau,\xi}  \in \mathbb{R}^{N N_t N_x \times N N_t  N_x}$ and space-time vectors $\underline{\mathbf{u}}, \underline{\mathbf{b}} \in \mathbb{R}^{N N_x N_t}$.

\subsection{Multigrid Solver}\label{Section:MGsolver}

In the next step, the linear system \eqref{MG} has to be solved. One method is to apply a forward substitution w.r.t. the time blocks. This involves inverting the diagonal blocks in each time step and results in a sequential process. Instead, we use a space-time multigrid method to solve \eqref{MG}. Multigrid algorithms consist of three main components: a smoother, intergrid transfer operators, i.e. prolongation and restriction, and a coarse grid solver. It is of advantage to study the intergrid transfer operators in space and time separately due to the special causality principle in the temporal direction. We consider a geometric multigrid method in space and time. We choose block Jacobi smoothers with blocks corresponding to one space-time slab since it has been shown in the case of space-time multigrid methods that pointwise or line-smoothers, when not chosen carefully, can result in divergent methods \cite{Hackbusch1984,Neumuller2013,Trottenberg2001}. The choice of these blocks follows from the block form of the full space-time system \eqref{MG}.

Let $\boldsymbol{\Omega}_{\ell} \subset \mathbb{R}^2$ be the grid on level $\ell=0,\dots,M$ with $\ell=0$ the coarsest and $\ell=M$ the finest level. We denote the number of time slabs on multigrid level $\ell$ by $N_{\ell_t}\in \mathbb{N}$ and the number of spatial elements by $N_{\ell_x}\in \mathbb{N}$. In consequence, on each space-time grid level $\ell$ the system matrix $\underline{\mathbf{L}}_{\tau_{\ell},\xi_{\ell}}$ is defined by \eqref{SpatialOperator} with $N_{\ell_x}$ volumes and the space-time system \eqref{MG} with $N_{\ell_t}$ time steps.

For the temporal component, restriction and prolongation matrices are defined using $\mathcal{L}_2$ projections,
\begin{align}\label{RandPtime}
\begin{split}
\mathbf{R}^{\ell_t}_{\ell_t-1} &:=
\begin{pmatrix}
\mathbf{R}_1 & \mathbf{R}_2 \\
& & \mathbf{R}_1 & \mathbf{R}_2\\
& & & \ddots & \ddots \\
& & & & \mathbf{R}_1 & \mathbf{R}_2
\end{pmatrix} \in  \mathbb{R}^{N_t N_{\ell_t-1} \times N_t N_{\ell_t}},\\
\mathbf{P}_{\ell_t}^{\ell_t-1} &:= (\mathbf{R}^{\ell_t}_{\ell_t-1})^T \in  \mathbb{R}^{N_t N_{\ell_t} \times N_t N_{\ell_t-1}},
\end{split}
\end{align}
with local prolongation matrices $\mathbf{R}_1^T := \mathbf{M}_{\tau_{\ell}}^{-1}\widetilde{\mathbf{M}}_{\tau_{\ell}}^1$ and $\mathbf{R}_2^T := \mathbf{M}_{\tau_{\ell}}^{-1}\widetilde{\mathbf{M}}_{\tau_{\ell}}^2$, see \cite{Gander2016}, \cite{Klaij2007}. For basis functions $\{\ell_k\}_{k=1}^{N_t} \subset \mathbb{P}^{p_t}(0,\tau_{\ell})$ on the fine grid and $\{\tilde{\ell}_k\}_{k=1}^{N_t} \subset \mathbb{P}^{p_t}(0,2\tau_{\ell})$ on the coarse grid, the local projection matrices from the coarse to the fine grid are defined by
\begin{align}
 \widetilde{M}_{\tau_{\ell}}^1(k,l) := \int_{0}^{\tau_{\ell}} \tilde{\ell}_l(t)\ell_k(t) dt,~~~ \widetilde{M}_{\tau_{\ell}}^2(k,l) := \int_{\tau_{\ell}}^{2\tau_{\ell}} \tilde{\ell}_l(t)\ell_k(t-\tau) dt,~ k,l=1,\dots,N_t.
\end{align}
Restriction and prolongation matrices in space are given by agglomeration
\begin{align}\label{RandPspace}
\begin{split}
 \mathbf{R}^{\ell_x}_{\ell_x-1} &:= \frac{1}{2}
 \begin{pmatrix}
  1 & 1 \\
  & & 1 & 1 \\
  & & & \ddots & \ddots \\
  & & & & 1 & 1
 \end{pmatrix} \in \mathbb{R}^{N_{\ell_x-1}\times N_{\ell_x}},\\
 \mathbf{P}_{\ell_x}^{\ell_x-1} &:= (2\mathbf{R}^{\ell_x}_{\ell_x-1})^T  \in \mathbb{R}^{N_{\ell_x}\times N_{\ell_x-1}}.
\end{split}
\end{align}

We study two different coarsening strategies: coarsening in space and time, referred to as full-coarsening and denoted by index $f$, and coarsening in the temporal direction only, which we refer to as semi-coarsening with index $s$. Since we are especially interested in the efficiency of the multigrid method in time, we study a semi-coarsening strategy in this direction.

For the space-time system, restriction and prolongation operators can then be defined with a tensor product
\begin{align}
 (\underline{\mathbf{R}}^{\ell}_{\ell-1})^s &:= \mathbf{I}_{N_{\ell_x}} \otimes \mathbf{R}^{\ell_t}_{\ell_t-1},~ (\underline{\mathbf{R}}^{\ell}_{\ell-1})^f := \mathbf{R}^{\ell_x}_{\ell_x-1} \otimes \mathbf{R}^{\ell_t}_{\ell_t-1}, \label{RestrictionM} \\
 (\underline{\mathbf{P}}_{\ell}^{\ell-1})^s &:= \mathbf{I}_{N_{\ell_x}} \otimes \mathbf{P}_{\ell_t}^{\ell_t-1},~ (\underline{\mathbf{P}}_{\ell}^{\ell-1})^f := \mathbf{P}_{\ell_x}^{\ell_x-1} \otimes \mathbf{P}_{\ell_t}^{\ell_t-1}. \label{ProlongationM}
\end{align}

As smoother we choose a damped block Jacobi method
\begin{align}\label{DampedBlockJacobi}
 \underline{\mathbf{u}}^{(k+1)}= \omega_t(\underline{\mathbf{D}}_{\tau_{\ell},\xi_{\ell}})^{-1} \underline{\mathbf{b}} + (\underline{\mathbf{I}}_{\tau_{\ell},\xi_{\ell}}-\omega_t(\underline{\mathbf{D}}_{\tau_{\ell},\xi_{\ell}})^{-1}\underline{\mathbf{L}}_{\tau_{\ell},\xi_{\ell}})\underline{\mathbf{u}}^{(k)},
\end{align}
with damping factor $\omega_t$ and block diagonal matrix
\begin{align}
\underline{\mathbf{D}}_{\tau_{\ell},\xi_{\ell}}:=\text{diag}( [\mathbf{A}_{\tau_{\ell},\xi_{\ell}},\dots,\mathbf{A}_{\tau_{\ell},\xi_{\ell}} ]).
\end{align}
Here, the blocks correspond to a space-time slab on the given grid level due to the block form of the full space-time system \eqref{MG}. The block Jacobi iteration matrix reads
\begin{align}\label{ItMatrixDampedJacobi}
 \underline{\mathbf{S}}_{\tau_{\ell},\xi_{\ell}} := \underline{\mathbf{I}}-\omega_t(\underline{\mathbf{D}}_{\tau_{\ell},\xi_{\ell}})^{-1}\underline{\mathbf{L}}_{\tau_{\ell},\xi_{\ell}}.
\end{align}

With this, the iteration matrices for a two-grid V-cycle with $\nu_1$ pre- and $\nu_2$ post-smoothing steps on the fine grid are given by
\begin{align}
 \underline{\mathbf{M}}_{\tau_{\ell},\xi_{\ell}}^s &:= \underline{\mathbf{S}}_{\tau_{\ell},\xi_{\ell}}^{\nu_2} \left[ \underline{\mathbf{I}} - (\underline{\mathbf{P}}_{\ell}^{\ell-1})^s (\underline{\mathbf{L}}_{2\tau_{\ell},\xi_{\ell}})^{-1} (\underline{\mathbf{R}}^{\ell}_{\ell-1})^s \underline{\mathbf{L}}_{\tau_{\ell},\xi_{\ell}} \right] \underline{\mathbf{S}}_{\tau_{\ell},\xi_{\ell}}^{\nu_1}, \label{Full2GridMatrix}\\
 \underline{\mathbf{M}}_{\tau_{\ell},\xi_{\ell}}^f &:= \underline{\mathbf{S}}_{\tau_{\ell},\xi_{\ell}}^{\nu_2} \left[ \underline{\mathbf{I}} - (\underline{\mathbf{P}}_{\ell}^{\ell-1})^f (\underline{\mathbf{L}}_{2\tau_{\ell},2\xi_{\ell}})^{-1} (\underline{\mathbf{R}}^{\ell}_{\ell-1})^f \underline{\mathbf{L}}_{\tau_{\ell},\xi_{\ell}} \right] \underline{\mathbf{S}}_{\tau_{\ell},\xi_{\ell}}^{\nu_1}, \label{Semi2GridMatrix}
\end{align}
for semi-coarsening and full-coarsening respectively. Here, it is assumed that the systems are solved exactly on the coarse grid.

\section{Preliminaries}\label{SectionBasicsLFA}

In this section we discuss some preliminaries which are needed for the local Fourier analysis presented in the following sections.

\subsection{DG-SEM and Lobatto IIIC Methods}

We start with the temporal DG-SEM discretization \eqref{TemporalOperator}. Using an upwind flux, which is a flux to fulfill the temporal causality principle, the DG-SEM discretization in time for the linear test equation
\begin{align}
 u_t + \lambda u = 0,~ t \in [0,T],~ u(0)=u_0,~ \lambda \ge 0,
\end{align}
reads 
\begin{align}\label{DGSEMtime}
 (\mathbf{K}_{\tau}+\lambda \mathbf{M}_{\tau}) \mathbf{u}^{n+1} = \mathbf{C}_{\tau} \mathbf{u}^{n}.
\end{align}

One can show that the scheme is equivalent to a specific Runge-Kutta time-stepping method:
 
\begin{theorem}[\cite{Boom2015,Hairer2006}]\label{LobattoIIIC}
 The DG-SEM \eqref{DGSEMtime} with $p_t +1$ Legendre-Gauss-Lobatto nodes is equivalent to the $(p_t+1)$-stage Runge-Kutta scheme Lobatto IIIC.
\end{theorem}

Here, equivalence is referred to the solution of the unknowns in the end of each element assuming the resulting systems are solved exactly.
 
 With the help of Theorem \ref{LobattoIIIC} some stability results for the temporal discretization can be drawn.
 
 \begin{theorem}[\cite{Hairer2010,Jay2015}]\label{StabilityFunction}
  For $s \in \mathbb{N}$ the $s$-stage Lobatto IIIC scheme is of order $2s-1$ and its stability function $R(z)$ is given by the $(s-2,s)$-Pad\'e approximation of the exponential function $e^z$. The method is $L$-stable and furthermore algebraically stable, thus B-stable and A-stable.
 \end{theorem}
 
 \begin{corollary}\label{stabilityFunctionDG}
  The stability function $R(z)$ of the DG-SEM \eqref{DGSEMtime} with $p_t +1 \in \mathbb{N}$ Legendre-Gauss-Lobatto nodes, $p_t \ge 1$, is given by the $(p_t-1,p_t+1)$-Pad\'e approximation to the exponential function $e^z$.
 \end{corollary}
 
 \begin{proof}
 By Theorem \ref{LobattoIIIC} DG-SEM is equivalent to the Lobatto IIIC method. Thus, both methods have the same stability function $R(z)$. By Theorem \ref{StabilityFunction} the stability function is given by the $(p_t-1,p_t+1)$-Pad\'e approximation to the exponential function $e^z$.
 \end{proof}
 
 The Pad\'e approximant for the exponential function can be calculated directly.
 
 \begin{theorem}[\cite{Hairer2010}]
  The $(k,m)$-\textit{Pad\'e approximant}
  \begin{align*}
   r_{km}(z) = \frac{p_{km}(z)}{q_{km}(z)}
  \end{align*}
  of the exponential function $e^z$ is given by
  \begin{align*}
   p_{km}(z) & = 1 + \sum_{j=1}^k \frac{(k+m-j)!\,k!}{(k+m)!\,(k-j)!} \cdot \frac{z^j}{j!},\\
   q_{km}(z) & = 1 + \sum_{j=1}^m \frac{(k+m-j)!\,m!}{(k+m)!\,(m-j)!} \cdot \frac{(-z)^j}{j!}.
  \end{align*}
 \end{theorem}
 
 With the help of the stability function $R$, the eigenvalues of the discretization matrix \eqref{DGSEMtime} can be calculated.
 
 \begin{lemma}[\cite{Gander2016}]\label{Eigenvalues}
 For $\lambda \in \mathbb{C}$ the spectrum of the matrix $(\mathbf{K}_{\tau} + \lambda \mathbf{M}_{\tau}) ^{-1}\mathbf{C}_{\tau}  \in \mathbb{C}^{N_t \times N_t}$ is given by
 \begin{align*}
  \sigma((\mathbf{K}_{\tau} + \lambda \mathbf{M}_{\tau})^{-1}\mathbf{C}_{\tau} ) = \{ 0,R(-\lambda \tau) \}
 \end{align*}
 where $R(z)$ is the stability function of the DG time stepping scheme, see Corollary \ref{stabilityFunctionDG}. 
 \end{lemma}
 
 These results are used for the smoothing analysis in section \ref{SectionSmoothingAnalysis}.

\subsection{Definitions and Notation for the Local Fourier Analysis}

In this section we present the basic tools needed to perform a local Fourier analysis for the multigrid solver as presented in section \ref{Section:MGsolver}. For a more detailed description of this technique we refer to \cite{Birken2021,Gustafsson2007}. 

First we define the Fourier modes and frequencies.

\begin{definition}[\cite{Wesseling2004}]
 The function
 \begin{align*}
 \boldsymbol{\varphi}(\theta_k) := [\varphi_1(\theta_k),\dots,\varphi_N(\theta_k)]^T,~ \varphi_j(\theta_k):= e^{\mathrm{i}j\theta_k},~j=1,\dots,N,~ N \in \mathbb{N},
 \end{align*}
 is called \textit{Fourier mode with frequencies}
 \begin{align*}
  \theta_k \in \Theta := \left\{ \frac{2k\pi}{N}: k= 1-\frac{N}{2},\dots,\frac{N}{2} \right\} \subset (-\pi,\pi].
 \end{align*}
 The frequencies can be separated into \textit{low and high frequencies}
 \begin{align*}
  \Theta^{low} := \Theta \cap \left(-\frac{\pi}{2},\frac{\pi}{2} \right],~\Theta^{high} := \Theta \cap \left(\left(-\pi,-\frac{\pi}{2}\right] \cup \left( \frac{\pi}{2},\pi \right] \right].
 \end{align*}
\end{definition}

In this paper we consider frequencies on a two-dimensional space-time domain. Given the set of space-time frequencies 
 \begin{align*}
  \Theta_{\ell_x,\ell_t} := \{ (\theta_x,\theta_t): \theta_x \in \Theta_{\ell_x},~\theta_t \in \Theta_{\ell_t} \} \subset (-\pi,\pi]^2,
 \end{align*}
 low and high frequencies are defined as
 \begin{align}
 \Theta_{\ell_x,\ell_t}^{high,s} &:= \Theta_{\ell_x,\ell_t} \setminus \Theta_{\ell_x,\ell_t}^{low,s} ~~\text{ for }~~ \Theta_{\ell_x,\ell_t}^{low,s} := \Theta_{\ell_x,\ell_t} \cap (-\pi,\pi] \times \left( -\frac{\pi}{2},\frac{\pi}{2} \right], \label{FrequenciesSemi}\\
 \Theta_{\ell_x,\ell_t}^{high,f} &:= \Theta_{\ell_x,\ell_t} \setminus \Theta_{\ell_x,\ell_t}^{low,f} ~~\text{ for }~~ \Theta_{\ell_x,\ell_t}^{low,f} := \Theta_{\ell_x,\ell_t} \cap \left( -\frac{\pi}{2},\frac{\pi}{2} \right]^2, \label{FrequenciesFull}
 \end{align}
 for semi-coarsening in time and full space-time coarsening respectively. In Figure \ref{CoarseningStrategies} the ranges for the frequencies in the space-time domain are visualized for both coarsening strategies.
 
 \begin{figure}
 \centering
 \begin{subfigure}{.4\linewidth}
 \begin{tikzpicture}
 \draw[fill=gray!30,gray!30] (0,0) -- (4,0) -- (4,1) -- (0,1);
 \draw[fill=gray!30,gray!30] (0,3) -- (4,3) -- (4,4) -- (0,4);
 \draw[dashed] (0,0) -- (4,0);
 \draw (4,4) -- (0,4);
 \draw (4,4) -- (4,0);
 \draw[dashed] (0,0) -- (0,4);
 \draw (0,1) -- (4,1);
 \draw[dashed] (0,3) -- (4,3);
 \draw[->] (-1,2) -- (5,2) node[right] {$\theta_x$};
 \draw[->] (2,-1) -- (2,5) node[above] {$\theta_t$};
 
 \node at (-0.4,1.8) {$-\pi$};
 \node at (4.2,1.8) {$\pi$};
 \node at (2.25,-0.3) {-$\pi$};
 \node at (2.25,4.2) {$\pi$};
 \node at (2.25,0.7) {-$\frac{\pi}{2}$};
 \node at (2.25,3.25) {$\frac{\pi}{2}$};
 
 \node at (1,3.5) {$\Theta^{\text{high,s}}$};
 \node at (1,0.5) {$\Theta^{\text{high,s}}$};
 \node at (1,2.5) {$\Theta^{\text{low,s}}$};
 \end{tikzpicture}
 \end{subfigure}
 \hskip2em
 \begin{subfigure}{.4\linewidth}
 \begin{tikzpicture}
 \draw[fill=gray!30,gray!30] (0,0) -- (4,0) -- (4,1) -- (0,1);
 \draw[fill=gray!30,gray!30] (0,3) -- (4,3) -- (4,4) -- (0,4);
 \draw[fill=gray!30,gray!30] (0,1) -- (1,1) -- (1,3) -- (0,3);
 \draw[fill=gray!30,gray!30] (3,1) -- (4,1) -- (4,3) -- (3,3);
 \draw[dashed] (0,0) -- (4,0);
 \draw (4,4) -- (0,4);
 \draw (4,4) -- (4,0);
 \draw[dashed] (0,0) -- (0,4);
 \draw (1,1) -- (3,1);
 \draw[dashed] (1,3) -- (3,3);
 \draw[dashed] (3,1) -- (3,3);
 \draw (1,1) -- (1,3);
 \draw[->] (-1,2) -- (5,2) node[right] {$\theta_x$};
 \draw[->] (2,-1) -- (2,5) node[above] {$\theta_t$};
 
 \node at (-0.4,1.8) {$-\pi$};
 \node at (4.2,1.8) {$\pi$};
 \node at (2.25,-0.3) {-$\pi$};
 \node at (2.25,4.2) {$\pi$};
 \node at (2.25,0.7) {-$\frac{\pi}{2}$};
 \node at (2.25,3.25) {$\frac{\pi}{2}$};
 \node at (0.75,1.7) {-$\frac{\pi}{2}$};
 \node at (3.25,1.7) {$\frac{\pi}{2}$};
 
 \node at (1,3.5) {$\Theta^{\text{high,f}}$};
 \node at (2,2.5) {$\Theta^{\text{low,f}}$};
 \end{tikzpicture} 
 \end{subfigure}
 \caption{Low and high frequencies for semi coarsening (left) and full coarsening (right)}\label{CoarseningStrategies}
 \end{figure}
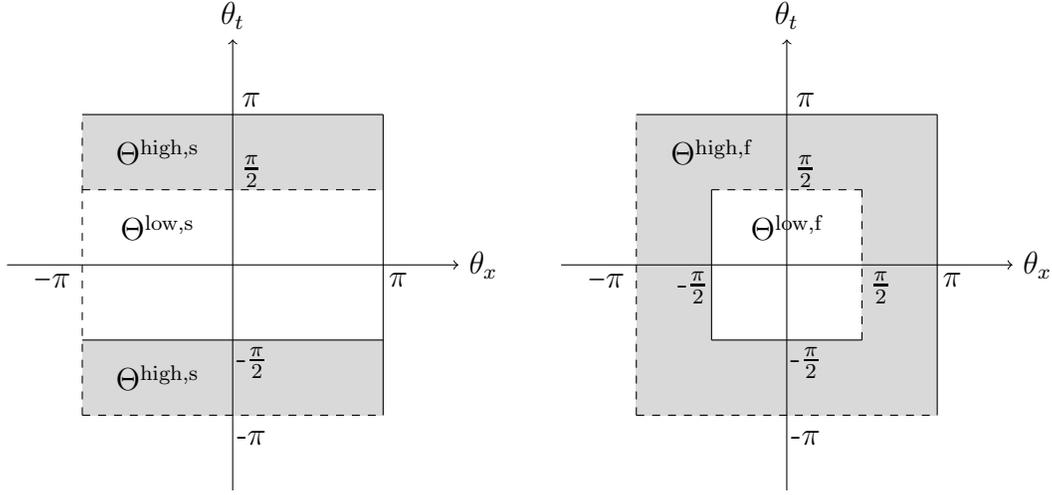

With this, the discrete Fourier transform reads:

\begin{theorem}[Discrete Fourier transform \cite{Wesseling2004}]\label{FourierTransform} 
 Let $\underline{\mathbf{u}} \in \mathbb{R}^{N_t N_{\ell_x} N_{\ell_t}}$ for $N_t,N_{\ell_x},N_{\ell_t} \in \mathbb{N}$, and assume that $N_{\ell_x}$ and $N_{\ell_t}$ are even. The vector $\underline{\mathbf{u}}$ can be represented as 
 \begin{align*}
  \underline{\mathbf{u}} = \sum_{\theta_x\in\Theta_{\ell_x}} \sum_{\theta_t\in\Theta_{\ell_t}} \underline{\boldsymbol{\psi}}(\theta_x,\theta_t),
 \end{align*}
 where $\underline{\boldsymbol{\psi}}(\theta_x,\theta_t) \in \mathbb{C}^{N_t N_{\ell_x} N_{\ell_t}}$ consists of the vectors
 \begin{align*}
  \boldsymbol{\psi}^n_{j}(\theta_x,\theta_t) := \mathbf{U}(\theta_x,\theta_t)\boldsymbol{\Phi}^n_{j}(\theta_x,\theta_t) \in \mathbb{C}^{N_t}, ~n=1,\dots,N_{\ell_t},~j=1,\dots,N_{\ell_x},
 \end{align*}
 and the vector $\boldsymbol{\Phi}^n_{j}(\theta_x,\theta_t) \in \mathbb{C}^{N_t}$ has elements 
 \begin{align*}
 \Phi^n_{j,l}(\theta_x,\theta_t) := \varphi_n(\theta_t)\varphi_j(\theta_x),~ l=1,\dots,N_t.
 \end{align*}
\end{theorem}
Moreover, we define the coefficient matrix as
 \begin{align*}
  \mathbf{U}(\theta_x,\theta_t) := \text{diag}(\hat{u}_1,\dots,\hat{u}_{N_t})\in \mathbb{C}^{N_t \times N_t},
 \end{align*}
 with coefficients
 \begin{align*}
  \hat{u}_l := \frac{1}{N_{\ell_x}}\frac{1}{N_{\ell_t}}\sum_{j=1}^{N_{\ell_x}}\sum_{n=1}^{N_{\ell_t}}u^n_{j,l}\varphi_j(-\theta_x)\varphi_n(-\theta_t),~l=1,\dots,N_t.
 \end{align*}
 
 Then, the linear space of Fourier modes can be defined.

\begin{definition}\label{FourierSpace}
 Consider the frequencies $\theta_x \in \Theta_{\ell_x}$ and $\theta_t \in \Theta_{\ell_t}$ and the vector $\boldsymbol{\Phi}^n_j(\theta_x,\theta_t)$ as in Theorem \ref{FourierTransform}. Then the \textit{linear space of Fourier modes} with frequencies $(\theta_x,\theta_t)$ is defined as
 \begin{align*}
  \Psi_{\ell_x,\ell_t}(\theta_x,\theta_t) &:= \text{span}\{ \underline{\boldsymbol{\Phi}}(\theta_x,\theta_t) \} \\
  & := \{ \underline{\boldsymbol{\psi}}(\theta_x,\theta_t) \in \mathbb{C}^{N_t \cdot N_{\ell_x} \cdot N_{\ell_t}}: \boldsymbol{\psi}^n_{j}(\theta_x,\theta_t) := \mathbf{U} \boldsymbol{\Phi}^n_{j}(\theta_x,\theta_l),\\
  &~~~~~~~~~~~~~~~~~~~\text{for}~n=1,\dots,N_{\ell_t},~j=1,\dots,N_{\ell_x}~\text{and}~\mathbf{U}\in\mathbb{C}^{N_t \times N_t}\}.
 \end{align*}
\end{definition}

 With the result from the next theorem it suffices to consider low frequencies.
 
 \begin{theorem}[\cite{Trottenberg2001}]\label{SumLowFreq}
  Let $\underline{\mathbf{u}} = [\mathbf{u}^1,\dots,\mathbf{u}^{N_{{\ell}_t}}]^T \in \mathbb{R}^{N_t N_{{\ell}_x} N_{{\ell}_t}}$ and assume that $N_{{\ell}_x}$ and $N_{\ell_t}$ are even numbers. Then the vector $\underline{\mathbf{u}}$ can be written as
  \begin{align*}
  \underline{ \mathbf{u}} = \sum_{(\theta_x,\theta_t)\in\Theta_{{\ell}_x,{\ell}_t}^{low,f}} \left( \underline{\boldsymbol{\psi}}(\theta_x,\theta_t) + \underline{\boldsymbol{\psi}}(\gamma(\theta_x),\theta_t) + \underline{\boldsymbol{\psi}}(\theta_x,\gamma(\theta_t)) + \underline{\boldsymbol{\psi}}(\gamma(\theta_x),\gamma(\theta_t)) \right),
  \end{align*}
 with the shifting operator 
 \begin{align*}
  \gamma(\theta) :=
  \begin{cases}
   \theta+\pi, & \theta <0,\\
   \theta-\pi, & \theta \ge 0,
  \end{cases}
 \end{align*}
 and $\underline{\boldsymbol{\psi}}(\theta_x,\theta_t) \in \mathbb{C}^{N_t N_{\ell_x} N_{\ell_t}}$ as in Lemma \ref{FourierTransform}.
 \end{theorem}
 
 Since $\underline{\boldsymbol{\psi}}(\theta_x,\theta_t)$ consists of the vectors $\boldsymbol{\psi}_j^n(\theta_x,\theta_t) = \mathbf{U}\boldsymbol{\Phi}_j^n(\theta_x,\theta_t)$, which itself build the vector $\underline{\boldsymbol{\Phi}}(\theta_x,\theta_t)$, the previous theorem implies that $\underline{\mathbf{u}} = [\mathbf{u}^1,\dots,\mathbf{u}^{N_{{\ell}_t}}]^T$ can be written as a linear combination of the low frequency vectors
 \begin{align*}
  \{ \underline{\boldsymbol{\Phi}}(\theta_x,\theta_t),\underline{\boldsymbol{\Phi}}(\gamma(\theta_x),\theta_t),\underline{\boldsymbol{\Phi}}(\theta_x,\gamma(\theta_t)),\underline{\boldsymbol{\Phi}}(\gamma(\theta_x),\gamma(\theta_t)) \}.
 \end{align*}
 Thus, four fine grid modes get aliased to one coarse grid mode. In the following it suffices therefore to only consider low frequencies and use the shifting operator $\gamma: \Theta_{\ell}^{low} \rightarrow \Theta_{\ell}^{high}$. We can therefore define a new Fourier space, based on low frequencies only:
 
 \begin{definition}\label{SpaceLowFrequencies} 
  For $N_t,N_{{\ell}_x},N_{{\ell}_t}$ consider the vector $\underline{\boldsymbol{\psi}}(\theta_x,\theta_t) \in \mathbb{C}^{N_t N_{{\ell}_x} N_{{\ell}_t}}$ for $(\theta_x,\theta_t)\in \Theta_{{\ell}_x,{\ell}_t}^{low,f}$ as in Lemma \ref{FourierTransform}. The \textit{linear space of low frequency harmonics} is defined as
  \begin{align*}
   \mathcal{E}_{{\ell}_x,{\ell}_t}(\theta_x,\theta_t) &:= \text{span}\{ \underline{\boldsymbol{\Phi}}(\theta_x,\theta_t),\underline{\boldsymbol{\Phi}}(\gamma(\theta_x),\theta_t),\underline{\boldsymbol{\Phi}}(\theta_x,\gamma(\theta_t)),\underline{\boldsymbol{\Phi}}(\gamma(\theta_x),\gamma(\theta_t)) \} \\
   & = \{ \underline{\boldsymbol{\psi}}(\theta_x,\theta_t) \in \mathbb{C}^{N_t \cdot  N_{{\ell}_x} \cdot N_{{\ell}_t}}: \boldsymbol{\psi}^n_{j}(\theta_x,\theta_t) = \mathbf{U}_1\boldsymbol{\Phi}^n_{j}(\theta_x,\theta_t)\\
   & ~~~~  + \mathbf{U}_2\boldsymbol{\Phi}^n_{j}(\gamma(\theta_x),\theta_t)+ \mathbf{U}_3\boldsymbol{\Phi}^n_{j}(\theta_x,\gamma(\theta_t)) + \mathbf{U}_4\boldsymbol{\Phi}^n_{j}(\gamma(\theta_x),\gamma(\theta_t)),\\
   & ~~~~~ n=1,\dots,N_{{\ell}_t},~j=1,\dots,N_{{\ell}_x}~\text{and}~\mathbf{U}_1,\mathbf{U}_2,\mathbf{U}_3,\mathbf{U}_4 \in \mathbb{C}^{N_t \times N_t}\}.
  \end{align*}
 \end{definition}
 
 Moreover, we can define the Fourier space for the semi-coarsening strategy.
 
  \begin{definition}[Fourier space, semi-coarsening]\label{FourierSymbSpaceSemi}
   For $N_t,N_{\ell_x},N_{\ell_t-1} \in \mathbb{N}$ and $(\theta_x,\theta_t)\in \Theta_{\ell_x,\ell_t}^{low,f}$ consider $\underline{\boldsymbol{\Phi}}(\theta_x,\theta_t)\in \mathbb{C}^{N_t N_{\ell_x} N_{\ell_t-1}}$ as in Lemma \ref{FourierTransform}. We define the linear space with frequencies $(\theta_x,2\theta_t)$ as
   \begin{align*}
    \Psi_{\ell_x,\ell_{t}-1}(\theta_x,2\theta_t) & := \text{span}\{ \underline{\boldsymbol{\Phi}}^{\ell_x,\ell_t-1}(\theta_x,2\theta_t),\underline{\boldsymbol{\Phi}}^{\ell_x,\ell_t-1}(\gamma(\theta_x),2\theta_t) \}\\
    & = \{ \underline{\boldsymbol{\psi}}^{\ell_x,\ell_t-1}(\theta_x,2\theta_t) \in \mathbb{C}^{N_t,N_{\ell_x},N_{\ell_t-1}}:\\
    & ~~~~~ \boldsymbol{\psi}^{n,\ell_x,\ell_t-1}_{j}(\theta_x,2\theta_t) = \mathbf{U}_1 \boldsymbol{\Phi}^{n,\ell_x,\ell_t-1}_{j}(\theta_x,2\theta_t)\\
    & ~~~~~ + \mathbf{U}_2 \boldsymbol{\Phi}^{n,\ell_x,\ell_t-1}_{j}(\gamma(\theta_x),2\theta_t) \text{ for } n=1,\dots,N_{\ell_t}-1,\\
    & ~~~~~ j=1,\dots,N_{\ell_x},~\mathbf{U}_1,\mathbf{U}_2 \in \mathbb{C}^{N_t \times N_t} \}.
   \end{align*}
  \end{definition}
  
 One key property of the LFA is the shifting equality, which is used extensively when deriving the Fourier symbols of the operators in the next section.

 \begin{lemma}[\cite{Neumuller2013}]\label{ShiftingEquality}
 Let $\theta_x \in \Theta_{\ell_x}$, $\theta_t \in \Theta_{\ell_t}$ and $\underline{\boldsymbol{\psi}}(\theta_x,\theta_t)\in\Psi_{\ell_x,\ell_t}(\theta_x,\theta_t)$. Then the following \textit{shifting equalities} hold:
 \begin{align*}
  \boldsymbol{\psi}^{n-1}_j(\theta_x,\theta_t) = e^{-\mathrm{i}\theta_t}\boldsymbol{\psi}^n_{j}(\theta_x,\theta_t),~ n=2,\dots,N_{\ell_t}\\
  \boldsymbol{\psi}^n_{j-1}(\theta_x,\theta_t) = e^{-\mathrm{i}\theta_x}\boldsymbol{\psi}^n_{j}(\theta_x,\theta_t),~j=2,\dots,N_{\ell_x}.
 \end{align*}
 \end{lemma}

\section{Fourier Symbols}\label{FourierSymbols}

The first step of the local Fourier analysis is to derive the Fourier symbols of all operators in the MG iteration \eqref{Semi2GridMatrix} and \eqref{Full2GridMatrix}, i.e. of the system matrix, smoother, restriction and prolongation. These symbols are also referred to as formal eigenvalues \cite{Trottenberg2001} since they are derived by multiplying the operators by the vector $\underline{\boldsymbol{\psi}}(\theta_x,\theta_t)$ from Theorem \ref{FourierTransform}.

We start with the Fourier symbol of the system matrix \eqref{MG}.

\begin{lemma}[Fourier symbol of $\underline{\mathbf{L}}_{\tau_{\ell},\xi_{\ell}}$]\label{FourierSymbolL}
For $\theta_x \in \Theta_{\ell_x}$ and $\theta_t \in \Theta_{\ell_t}$ we consider the vector $\underline{\boldsymbol{\psi}}(\theta_x,\theta_t)\in \Psi_{\ell_x,\ell_t}(\theta_x,\theta_t)$. For
 \begin{align*}
  \boldsymbol{\mathcal{L}}_{\tau_{\ell},\xi_{\ell}} (\theta_x,\theta_t):= -e^{-\mathrm{i}\theta_t}\mathbf{C}_{\tau_{\ell}}+\mathbf{K}_{\tau_{\ell}}+\frac{a}{\Delta x} (-e^{-\mathrm{i}\theta_x}+1)\mathbf{M}_{\tau_{\ell}} \in \mathbb{C}^{N_t \times N_t}
 \end{align*}
 it holds that
\begin{align*}
  (\underline{\mathbf{L}}_{\tau_{\ell},\xi_{\ell}}\underline{\boldsymbol{\psi}}(\theta_x,\theta_t))^n_{j} = \boldsymbol{\mathcal{L}}_{\tau_{\ell},\xi_{\ell}}(\theta_x,\theta_t) \boldsymbol{\psi}^n_{j}(\theta_x,\theta_t),
 \end{align*}
 for $n=2,\dots,N_{\ell_t},~j=2,\dots,N_{\ell_x}-1$ and we call $\boldsymbol{\mathcal{L}}_{\tau_{\ell},\xi_{\ell}} (\theta_x,\theta_t) \in \mathbb{C}^{N_t \times N_t}$ the Fourier symbol of $\underline{\mathbf{L}}_{\tau_{\ell},\xi_{\ell}} \in \mathbb{C}^{N_{\ell_t} N_t N_{\ell_x} \times N_{\ell_t} N_t N_{\ell_x}}$.
\end{lemma}

\begin{proof}
 With Lemma \ref{ShiftingEquality} we get for $\underline{\boldsymbol{\psi}}(\theta_x,\theta_t) \in \Psi_{\ell_x,\ell_t}(\theta_x,\theta_t)$
 \begin{align*}
  (\underline{\mathbf{L}}_{\tau_{\ell},\xi_{\ell}}\underline{\boldsymbol{\psi}}(\theta_x,\theta_t))^{n} & = \mathbf{B}_{\tau_{\ell},\xi_{\ell}}\boldsymbol{\psi}^{n-1}(\theta_x,\theta_t) + \mathbf{A}_{\tau_{\ell},\xi_{\ell}}\boldsymbol{\psi}^{n}(\theta_x,\theta_t) \\
  & = (e^{-\mathrm{i}\theta_t}\mathbf{B}_{\tau_{\ell},\xi_{\ell}} + \mathbf{A}_{\tau_{\ell},\xi_{\ell}})\boldsymbol{\psi}^{n}(\theta_x,\theta_t),~ n=2,\dots,N_{\ell_t}.
 \end{align*}
 Thus, we have to study the product of $\mathbf{A}_{\tau_{\ell},\xi_{\ell}} = \mathbf{I}_{\xi_{\ell}} \otimes \mathbf{K}_{\tau_{\ell}} + \mathbf{K}_{\xi_{\ell}} \otimes \mathbf{M}_{\tau_{\ell}}$ and $\mathbf{B}_{\tau_{\ell},\xi_{\ell}} = - \mathbf{I}_{\xi_{\ell}} \otimes \mathbf{C}_{\tau_{\ell}}$ with the vector $\boldsymbol{\psi}^{n}(\theta_x,\theta_t)$:
 \begin{align*}
  (\mathbf{A}_{\tau_{\ell},\xi_{\ell}}\boldsymbol{\psi}^{n}(\theta_x,\theta_t))_{j,l} & = \sum_{i=1}^{N_{{\ell}_x}} \sum_{k=1}^{N_t} I_{\xi_{\ell}}(j,i) K_{\tau_{\ell}}(l,k) \psi^n_{i,k}(\theta_x,\theta_t) \\
  & ~ + \sum_{i=1}^{N_{h_{\ell}}} \sum_{k=1}^{N_t} K_{\xi_{\ell}}(j,i) M_{\tau_{\ell}}(l,k) \psi^n_{i,k}(\theta_x,\theta_t)\\
  & = (\mathbf{K}_{\tau_{\ell}}\boldsymbol{\psi}^n_{j}(\theta_x,\theta_t))_l + \frac{a}{\Delta x} (-e^{-\mathrm{i}\theta_x}+1) (\mathbf{M}_{\tau_{\ell}}\boldsymbol{\psi}^n_{j}(\theta_x,\theta_t))_l \\
  & = (\mathbf{K}_{\tau_{\ell}} + \frac{a}{\Delta x} (-e^{-\mathrm{i}\theta_x}+1) \mathbf{M}_{\tau_{\ell}})\boldsymbol{\psi}^n_{j}(\theta_x,\theta_t))_l,
 \end{align*}
 and
  \begin{align*}
  (\mathbf{B}_{\tau_{\ell},\xi_{\ell}}\boldsymbol{\psi}^{n}(\theta_x,\theta_t))_{j,l} & = - \sum_{i=1}^{N_{{\ell}_x}} \sum_{k=1}^{N_t} I_{\xi_{\ell}}(j,i) C_{\tau_{\ell}}(l,k) \psi^n_{i,k}(\theta_x,\theta_t)\\
  & = - \sum_{k=1}^{N_t} C_{\tau_{\ell}}(l,k) \psi^n_{j,k}(\theta_x,\theta_t)
   = -(\mathbf{C}_{\tau_{\ell}}\boldsymbol{\psi}^n_{j}(\theta_x,\theta_t))_l,
 \end{align*}
  for $j=2,\dots,N_{{\ell}_x}-1$ and $l=1,\dots,N_t$. Then
 \begin{align*}
  (\underline{\mathbf{L}}_{\tau_{\ell},\xi_{\ell}}\underline{\boldsymbol{\psi}}(\theta_x,\theta_t))^n_{j} = (-e^{-\mathrm{i}\theta_t}\mathbf{C}_{\tau_{\ell}}+\mathbf{K}_{\tau_{\ell}}+\frac{a}{\Delta x} (-e^{-\mathrm{i}\theta_x}+1)\mathbf{M}_{\tau_{\ell}})\boldsymbol{\psi}^n_{j}(\theta_x,\theta_t),
 \end{align*} 
 and thus
  \begin{align*}
   \boldsymbol{\mathcal{L}}_{\tau_{\ell},\xi_{\ell}} = -e^{-\mathrm{i}\theta_t}\mathbf{C}_{\tau_{\ell}}+\mathbf{K}_{\tau_{\ell}}+\frac{a}{\Delta x} (-e^{-\mathrm{i}\theta_x}+1)\mathbf{M}_{\tau_{\ell}} \in \mathbb{C}^{N_t \times N_t}.
  \end{align*}
\end{proof}

With this result we can derive the symbol of the the block Jacobi smother \eqref{DampedBlockJacobi}.
 
 \begin{lemma}[Fourier symbol of $\underline{\mathbf{S}}_{\tau_{\ell},\xi_{\ell}}$]\label{FourierSymbolS}
  For $\theta_x \in \Theta_{{\ell}_x}$ and $\theta_t \in \Theta_{{\ell}_t}$ we consider the vector $\underline{\boldsymbol{\psi}}(\theta_x,\theta_t) \in \Psi_{{\ell}_x,{\ell}_t}(\theta_x,\theta_t)$. For 
 \begin{align*}
  \boldsymbol{\mathcal{S}}_{\tau_{\ell},\xi_{\ell}}(\theta_x,\theta_t) := (1-\omega_t)\mathbf{I}_{N_t}+\omega_t e^{-\mathrm{i}\theta_t}(\mathbf{K}_{\tau_{\ell}}+\frac{a}{\Delta x} (-e^{-\mathrm{i}\theta_x}+1)\mathbf{M}_{\tau_{\ell}})^{-1}\mathbf{C}_{\tau_{\ell}} \in \mathbb{C}^{N_t \times N_t}
 \end{align*}
 it holds that
  \begin{align*}
   (\underline{\mathbf{S}}_{\tau_{\ell},\xi_{\ell}}\underline{\boldsymbol{\psi}}(\theta_x,\theta_t))^n_{j} = \boldsymbol{\mathcal{S}}_{\tau_{\ell},\xi_{\ell}}(\theta_x,\theta_t) \boldsymbol{\psi}^n_{j}(\theta_x,\theta_t)
  \end{align*}
  for $n=1,\dots,N_{{\ell}_t}$, $j=1,\dots,N_{{\ell}_x}$ and we call $\boldsymbol{\mathcal{S}}_{\tau_{\ell},\xi_{\ell}}(\theta_x,\theta_t) \in \mathbb{C}^{N_t \times N_t}$ the Fourier symbol of $\underline{\mathbf{S}}_{\tau_{\ell},\xi_{\ell}} \in \mathbb{C}^{N_{\ell_t} N_t N_{\ell_x} \times N_{\ell_t} N_t N_{\ell_x}}$.
 \end{lemma}
 
 \begin{proof}
  Let $\underline{\boldsymbol{\psi}}(\theta_x,\theta_t) \in \Psi_{{\ell}_x,{\ell}_t}(\theta_x,\theta_t)$. For fixed $n=1,\dots,N_{{\ell}_t}$ and $j=1,\dots,N_{{\ell}_x}$ it holds
  \begin{align*}
    (\underline{\mathbf{S}}_{\tau_{\ell},\xi_{\ell}}\underline{\boldsymbol{\psi}}(\theta_x,\theta_t))^n_{j} & = ((\underline{\mathbf{I}}_{N_t N_{{\ell}_x} N_{{\ell}_t}} - \omega_t(\underline{\mathbf{D}}_{\tau_{\ell},\xi_{\ell}})^{-1}\underline{\mathbf{L}}_{\tau_{\ell},\xi_{\ell}})\underline{\boldsymbol{\psi}}(\theta_x,\theta_t))^n_{j}\\
    & = (\mathbf{I}_{N_t}-\omega_t(\hat{\mathbf{A}}_{\tau_{\ell},\xi_{\ell}}(\theta_x))^{-1}\boldsymbol{\mathcal{L}}_{\tau_{\ell},\xi_{\ell}}(\theta_x,\theta_t))\boldsymbol{\psi}^n_{j}(\theta_x,\theta_t)\\
    & = \boldsymbol{\mathcal{S}}_{\tau_{\ell},\xi_{\ell}}(\theta_x,\theta_t)\boldsymbol{\psi}^n_{j}(\theta_x,\theta_t),
  \end{align*}
  with $\hat{\mathbf{A}}_{\tau_{\ell},\xi_{\ell}}(\theta_x) := \mathbf{K}_{\tau_{\ell}} + \frac{a}{\Delta x} (-e^{-\mathrm{i}\theta_x}+1)\mathbf{M}_{\tau_{\ell}}$ derived as in the previous proof. Moreover,
  \begin{align*}
   (\hat{\mathbf{A}}_{\tau_{\ell},\xi_{\ell}}(\theta_x))^{-1}\boldsymbol{\mathcal{L}}_{\tau_{\ell},\xi_{\ell}}(\theta_x,\theta_t) & = (\mathbf{K}_{\tau_{\ell}} + \frac{a}{\Delta x} (-e^{-\mathrm{i}\theta_x}+1)\mathbf{M}_{\tau_{\ell}})^{-1}\\
   & ~~~~ (-e^{-\mathrm{i}\theta_t}\mathbf{C}_{\tau_{\ell}}+\mathbf{K}_{\tau_{\ell}}+\frac{a}{\Delta x} (-e^{-\mathrm{i}\theta_x}+1)\mathbf{M}_{\tau_{\ell}})\\
   & = \mathbf{I}_{N_t} - e^{-\mathrm{i}\theta_t}(\mathbf{K}_{\tau_{\ell}}+\frac{a}{\Delta x}(-e^{-\mathrm{i}\theta_x}+1)\mathbf{M}_{\tau_{\ell}})^{-1}\mathbf{C}_{\tau_{\ell}}.
  \end{align*}
  Thus,
  \begin{align*}
    \boldsymbol{\mathcal{S}}_{\tau_{\ell},\xi_{\ell}}(\theta_x,\theta_t) & = \mathbf{I}_{N_t} - \omega_t(\mathbf{I}_{N_t} - e^{-\mathrm{i}\theta_t}(\mathbf{K}_{\tau_{\ell}}+\frac{a}{\Delta x}(-e^{-\mathrm{i}\theta_x}+1)\mathbf{M}_{\tau_{\ell}})^{-1}\mathbf{C}_{\tau_{\ell}})\\
    & = (1-\omega_t)\mathbf{I}_{N_t} + \omega_t e^{-\mathrm{i}\theta_t}(\mathbf{K}_{\tau_{\ell}}+\frac{a}{\Delta x}(-e^{-\mathrm{i}\theta_x}+1)\mathbf{M}_{\tau_{\ell}})^{-1}\mathbf{C}_{\tau_{\ell}} \in \mathbb{C}^{N_t \times N_t}.
  \end{align*}
 \end{proof}
 
 With Theorems \ref{FourierTransform} and \ref{SumLowFreq} and Lemma \ref{FourierSymbolL} we get for the system matrix $\underline{\mathbf{L}}_{\tau_{\ell},\xi_{\ell}}$ and $(\theta_x,\theta_t)\in \Theta_{{\ell}_x,{\ell}_t}^{low,f}$ the following mapping property:
 \begin{align}\label{MappingSystemMatrix}
 \begin{split}
  \underline{\mathbf{L}}_{\tau_{\ell},\xi_{\ell}}: & \mathcal{E}_{{\ell}_x,{\ell}_t}(\theta_x,\theta_t) \to \mathcal{E}_{{\ell}_x,{\ell}_t}(\theta_x,\theta_t),\\
  &\begin{pmatrix}
   \mathbf{U}_1\\
   \mathbf{U}_2\\
   \mathbf{U}_3\\
   \mathbf{U}_4
  \end{pmatrix}
  \to
  \begin{pmatrix}
   \boldsymbol{\mathcal{L}}_{\tau_{\ell},\xi_{\ell}}(\theta_x,\theta_t)\mathbf{U}_1\\
   \boldsymbol{\mathcal{L}}_{\tau_{\ell},\xi_{\ell}}(\gamma(\theta_x),\theta_t)\mathbf{U}_2\\
   \boldsymbol{\mathcal{L}}_{\tau_{\ell},\xi_{\ell}}(\theta_x,\gamma(\theta_t))\mathbf{U}_3\\
   \boldsymbol{\mathcal{L}}_{\tau_{\ell},\xi_{\ell}}(\gamma(\theta_x),\gamma(\theta_t))\mathbf{U}_4\\
  \end{pmatrix}
  =: \widetilde{\boldsymbol{\mathcal{L}}}_{\tau_{\ell},\xi_{\ell}}(\theta_x,\theta_t)
  \begin{pmatrix}
   \mathbf{U}_1\\
   \mathbf{U}_2\\
   \mathbf{U}_3\\
   \mathbf{U}_4
  \end{pmatrix},
 \end{split}
 \end{align}
 with a block diagonal matrix $\widetilde{\boldsymbol{\mathcal{L}}}_{\tau_{\ell},\xi_{\ell}}(\theta_x,\theta_t) \in \mathbb{C}^{4N_t \times 4N_t}$, $\boldsymbol{\mathcal{L}}_{\tau_{\ell},\xi_{\ell}} \in \mathbb{C}^{N_t \times N_t}$ as defined in Lemma \ref{FourierSymbolL} and the space of low frequencies $\mathcal{E}_{{\ell}_x,{\ell}_t}$ as defined in \ref{SpaceLowFrequencies}. Accordingly, we obtain with Lemma \ref{FourierSymbolS} for the smoother $ \underline{\mathbf{S}}_{\tau_{\ell},\xi_{\ell}}$ and $(\theta_x,\theta_t)\in\Theta_{{\ell}_x,{\ell}_t}^{low,f}$
 \begin{align}\label{MappingSmoother}
 \begin{split}
  \underline{\mathbf{S}}_{\tau_{\ell},\xi_{\ell}}: & \mathcal{E}_{{\ell}_x,{\ell}_t}(\theta_x,\theta_t) \to \mathcal{E}_{{\ell}_x,{\ell}_t}(\theta_x,\theta_t),\\
  & \begin{pmatrix}
   \mathbf{U}_1\\
   \mathbf{U}_2\\
   \mathbf{U}_3\\
   \mathbf{U}_4
  \end{pmatrix}
  \to
  \begin{pmatrix}
   \boldsymbol{\mathcal{S}}_{\tau_{\ell},\xi_{\ell}}(\theta_x,\theta_t)\mathbf{U}_1\\
   \boldsymbol{\mathcal{S}}_{\tau_{\ell},\xi_{\ell}}(\gamma(\theta_x),\theta_t)\mathbf{U}_2\\
   \boldsymbol{\mathcal{S}}_{\tau_{\ell},\xi_{\ell}}(\theta_x,\gamma(\theta_t))\mathbf{U}_3\\
   \boldsymbol{\mathcal{S}}_{\tau_{\ell},\xi_{\ell}}(\gamma(\theta_x),\gamma(\theta_t))\mathbf{U}_4\\
  \end{pmatrix}
   =: \widetilde{\boldsymbol{\mathcal{S}}}_{\tau_{\ell},\xi_{\ell}}(\theta_x,\theta_t)
  \begin{pmatrix}
   \mathbf{U}_1\\
   \mathbf{U}_2\\
   \mathbf{U}_3\\
   \mathbf{U}_4
  \end{pmatrix},
 \end{split}
 \end{align}
 with a block diagonal matrix $\widetilde{\boldsymbol{\mathcal{S}}}_{\tau_{\ell},\xi_{\ell}}(\theta_x,\theta_t) \in \mathbb{C}^{4N_t \times 4N_t}$ and $\boldsymbol{\mathcal{S}}_{\tau_{\ell},\xi_{\ell}} \in \mathbb{C}^{N_t \times N_t}$ as defined in Lemma \ref{FourierSymbolS}.
 
 Next, we derive the Fourier symbols of the restriction and prolongation operators.
 
 \begin{lemma}[Fourier symbols of spatial prolongation and restriction]\label{FourierSymbPandRspace}
  Consider the spatial restriction and prolongation operators $\mathbf{R}_{\ell_x-1}^{\ell_x} \in \mathbb{C}^{N_{{\ell}_x-1} \times N_{{\ell}_x}}$ and $\mathbf{P}^{\ell_x-1}_{\ell_x}  \in \mathbb{C}^{N_{{\ell}_x} \times N_{{\ell}_x-1}}$ defined in \eqref{RandPspace}. Let $\boldsymbol{\varphi}^{\ell_x}(\theta_x)\in \mathbb{C}^{N_{{\ell}_x}}$ be a fine Fourier mode and $\boldsymbol{\varphi}^{\ell_x-1}(2\theta_x)\in \mathbb{C}^{N_{{\ell}_x-1}}$ a coarse Fourier mode for $\theta_x \in \Theta_{{\ell}_x}^{low}$. Then for $\mathcal{R}_{\ell_x-1}^{\ell_x}(\theta_x) := \frac{1}{2}(e^{-\mathrm{i}\theta_x}+1)$ it holds
  \begin{align*}
   (\mathbf{R}_{\ell_x-1}^{\ell_x}\boldsymbol{\varphi}^{\ell_x}(\theta_x))_j = \mathcal{R}_{\ell_x-1}^{\ell_x}(\theta_x)\varphi_j^{\ell_x-1}(2\theta_x), ~ j=1,\dots,N_{{\ell}_x-1},
  \end{align*}
 and we call $\mathcal{R}_{\ell_x-1}^{\ell_x}(\theta_x) \in \mathbb{C}$ the Fourier symbol of the restriction operator in space.\\
 For $\mathcal{P}^{\ell_x-1}_{\ell_x}(\theta_x) := \frac{1}{2}(e^{\mathrm{i}\theta_x}+1)$ it holds
  \begin{align*}
  (\mathbf{P}^{\ell_x-1}_{\ell_x}\boldsymbol{\varphi}^{\ell_x-1}(2\theta_x))_i = \mathcal{P}^{\ell_x-1}_{\ell_x}(\theta_x)\varphi_i^{\ell_x}(\theta_x) + \mathcal{P}^{\ell_x-1}_{\ell_x}(\gamma(\theta_x))\varphi_i^{\ell_x}(\gamma(\theta_x)), ~ i=1,\dots,N_{{\ell}_x}
  \end{align*}
 and we call $\mathcal{P}^{\ell_x-1}_{\ell_x}(\theta_x) \in \mathbb{C}$ the Fourier symbol of the prolongation operator in space.
 \end{lemma}
 
 \begin{proof}
  For the restriction operator we get
  \begin{align*}
   (\mathbf{R}_{\ell_x-1}^{\ell_x}\boldsymbol{\varphi}^{\ell_x}(\theta_x))_j &= \frac{1}{2} (\varphi^{\ell_x}_{2j-1}(\theta_x)+\varphi_{2j}^{\ell_x}(\theta_x)) 
   =  \frac{1}{2}(e^{-\mathrm{i}\theta_x}+1)\varphi_{2j}^{\ell_x}(\theta_x)\\
   & =  \frac{1}{2}(e^{-\mathrm{i}\theta_x}+1)\varphi_{j}^{\ell_x-1}(2\theta_x)
   =  \mathcal{R}_{\ell_x-1}^{\ell_x}(\theta_x)\varphi_j^{\ell_x-1}(2\theta_x),~j=1,\dots,N_{{\ell}_x-1},
  \end{align*}
  using the shifting Lemma \ref{ShiftingEquality} and $\varphi_{2j}^{\ell_x}(\theta_x)=\varphi_{j}^{\ell_x-1}(2\theta_x)$.\\
  For the prolongation operator it holds
  \begin{align*}
   (\mathbf{P}^{\ell_x-1}_{\ell_x}\boldsymbol{\varphi}^{\ell_x-1}(2\theta_x))_{2j-1} = \varphi^{\ell_x-1}_{j}(2\theta_x) = \varphi_{2j}^{\ell_x}(\theta_x) = e^{\mathrm{i}\theta_x}\varphi_{2j-1}^{\ell_x}(\theta_x),
  \end{align*}
  and
  \begin{align*}
   (\mathbf{P}^{\ell_x-1}_{\ell_x}\boldsymbol{\varphi}^{\ell_x-1}(2\theta_x))_{2j} = \varphi^{\ell_x-1}_{j}(2\theta_x) = \varphi_{2j}^{\ell_x}(\theta_x),
  \end{align*}
 for $j=1,\dots,N_{{\ell}_x-1}$, with the same arguments as before. Then
 \begin{align*}
  (\mathbf{P}^{\ell_x-1}_{\ell_x}\boldsymbol{\varphi}^{\ell_x-1}(2\theta_x))_{j} =
  \begin{cases}
   e^{\mathrm{i}\theta_x}\varphi_{j}^{\ell_x}(\theta_x),& j \text{ odd},\\
   \varphi_{j}^{\ell_x}(\theta_x),&j \text{ even},
  \end{cases}
 \end{align*}
 for $j=1,\dots,N_{{\ell}_x}$. Moreover,
 \begin{align*}
  \varphi_j^{\ell_x}(\gamma(\theta_x)) = e^{\mathrm{i}j\gamma(\theta_x)} =
  \begin{cases}
   e^{\mathrm{i}j\pi}e^{\mathrm{i}j\theta_x} ,& \theta_x < 0,\\
   e^{-\mathrm{i}j\pi}e^{\mathrm{i}j\theta_x} ,& \theta_x \ge 0,
  \end{cases}
  = 
  \begin{cases}
   -\varphi_j^{\ell_x}(\theta_x),& j \text{ odd},\\
   \varphi_j^{\ell_x}(\theta_x),& j \text{ even},
  \end{cases}
 \end{align*}
 and
 \begin{align*}
  \mathcal{P}^{\ell_x-1}_{\ell_x}(\gamma(\theta_x)) = \frac{1}{2}(e^{\mathrm{i}\gamma(\theta_x)}+1) =  \frac{1}{2}(-e^{\mathrm{i}\theta_x}+1),
 \end{align*}
 for $j=1,\dots,N_{{\ell}_x}$. This implies
 \begin{align*}
   & \mathcal{P}^{\ell_x-1}_{\ell_x} (\theta_x)\varphi_j^{\ell_x}(\theta_x) + \mathcal{P}^{\ell_x-1}_{\ell_x}(\gamma(\theta_x))\varphi_j^{\ell_x}(\gamma(\theta_x))\\
   & = \frac{1}{2}(e^{\mathrm{i}\theta_x}+1)\varphi_j^{\ell_x}(\theta_x) + \frac{1}{2}(-e^{\mathrm{i}\theta_x}+1)
   \begin{cases}
   -\varphi_j^{\ell_x}(\theta_x),& j \text{ odd},\\
   \varphi_j^{\ell_x}(\theta_x),& j \text{ even}
   \end{cases}\\
   & = 
   \begin{cases}
   e^{\mathrm{i}\theta_x}\varphi_{j}^{\ell_x}(\theta_x),& j \text{ odd},\\
   \varphi_{j}^{\ell_x}(\theta_x),& j \text{ even},
  \end{cases}
  = (\mathbf{P}^{\ell_x-1}_{\ell_x}\boldsymbol{\varphi}^{\ell_x-1}(2\theta_x))_{j},
 \end{align*}
 for $j=1,\dots,N_{\ell_x}$.
 \end{proof}
 
 The following five Lemmata from \cite{Gander2016} give us the Fourier symbols of the restriction and prolongation operators for the different coarsening strategies.
 
 \begin{lemma}[Fourier symbols for temporal prolongation and restriction]\label{FourierSymbPandRtime}
  Consider temporal restriction and prolongation operators $\mathbf{R}^{\ell_t}_{\ell_t-1} \in \mathbb{C}^{N_t N_{\ell_t-1} \times N_t N_{\ell_t}}$ and $\mathbf{P}^{\ell_t-1}_{\ell_t}  \in \mathbb{C}^{N_t N_{\ell_t} \times N_t N_{\ell_t-1}}$ as defined in \eqref{RandPtime}. Let $\boldsymbol{\Phi}^{\ell_t}(\theta_t)\in \mathbb{C}^{N_t N_{\ell_t}}$ be a fine Fourier mode and $\boldsymbol{\Phi}^{\ell_t-1}(\theta_t)\in \mathbb{C}^{N_t N_{\ell_t-1}}$ a coarse Fourier mode for $\theta_t \in \Theta_{\ell_t}^{low}$ with elements
  \begin{align*}
   \Phi^{n,\ell_t}_{l}(\theta_t) &:= \varphi_n(\theta_t),~~~ l=1,\dots,N_t,~n=1,\dots,N_{\ell_t},\\
   \Phi^{n,\ell_t-1}_{l}(\theta_t) &:= \varphi_{n}(\theta_t),~~~ l=1,\dots,N_t,~n=1,\dots,N_{\ell_t-1}.
  \end{align*}
  Then for $\boldsymbol{\mathcal{R}}^{\ell_t}_{\ell_t-1}(\theta_t):= e^{-\mathrm{i}\theta_t}\mathbf{R}_1 + \mathbf{R}_2 \in \mathbb{C}^{N_t \times N_t}$, with $\mathbf{R}_1$ and $\mathbf{R}_2$ defined in \eqref{RandPtime}, it holds
  \begin{align*}
   (\mathbf{R}^{\ell_t}_{\ell_t-1}\boldsymbol{\Phi}^{\ell_t}(\theta_t))^{n} = \boldsymbol{\mathcal{R}}^{\ell_t}_{\ell_t-1}(\theta_t)\boldsymbol{\Phi}^{n,\ell_t-1}(2\theta_t), ~ n=1,\dots,N_{\ell_t-1},
  \end{align*}
 and we call $\boldsymbol{\mathcal{R}}^{\ell_t}_{\ell_t-1}(\theta_t) \in \mathbb{C}^{N_t \times N_t}$ the  Fourier symbol for the restriction operator in time.\\
 Moreover, for $\boldsymbol{\mathcal{P}}^{\ell_t-1}_{\ell_t}(\theta_t) := \frac{1}{2}(e^{\mathrm{i}\theta_t}\mathbf{R}_1^T+\mathbf{R}_2^T) \in \mathbb{C}^{N_t \times N_t}$ it holds
  \begin{align*}
  (\mathbf{P}^{\ell_t-1}_{\ell_t}\boldsymbol{\Phi}^{\ell_t-1}(2\theta_t))^n = \boldsymbol{\mathcal{P}}^{\ell_t-1}_{\ell_t}(\theta_t)\boldsymbol{\Phi}^{n,\ell_t}(\theta_t) + \boldsymbol{\mathcal{P}}^{\ell_t-1}_{\ell_t}(\gamma(\theta_t))\boldsymbol{\Phi}^{n,\ell_t}(\gamma(\theta_t)),~ n=1,\dots,N_{\ell_t},
  \end{align*}
 and we call $\boldsymbol{\mathcal{P}}^{\ell_t-1}_{\ell_t}(\theta_t)  \in \mathbb{C}^{N_t \times N_t}$ the Fourier symbol for the prolongation in time.
 \end{lemma}

With these results we can get the mapping properties for the semi-restriction and semi-prolongation operators.
 
 \begin{lemma}[Fourier symbol for restriction, semi-coarsening]\label{FourierSymbRsemi}
 The following mapping property holds for the restriction operator $(\underline{\mathbf{R}}_{\ell-1}^{\ell})^s$:
  \begin{align*}
   (\underline{\mathbf{R}}_{\ell-1}^{\ell})^s: \mathcal{E}_{\ell_x,\ell_t}(\theta_x,\theta_t) &\to \Psi_{\ell_x,\ell_t-1}(\theta_x,2\theta_t),\\
   \begin{pmatrix}
    \mathbf{U}_1\\
    \mathbf{U}_2\\
    \mathbf{U}_3\\
    \mathbf{U}_4
   \end{pmatrix}
  &\mapsto
 ({\widetilde{\boldsymbol{\mathcal{R}}}}_{\ell-1}^{\ell})^s(\theta_t)
   \begin{pmatrix}
    \mathbf{U}_1\\
    \mathbf{U}_2\\
    \mathbf{U}_3\\
    \mathbf{U}_4
   \end{pmatrix},
  \end{align*}
  with
 \begin{align*}
  ({\widetilde{\boldsymbol{\mathcal{R}}}}_{\ell-1}^{\ell})^s(\theta_t) :=
 \begin{pmatrix}
  \boldsymbol{\mathcal{R}}_{\ell_t-1}^{\ell_t}(\theta_t) & 0 & \boldsymbol{\mathcal{R}}_{\ell_t-1}^{\ell_t}(\gamma(\theta_t)) & 0 \\
  0 & \boldsymbol{\mathcal{R}}_{\ell_t-1}^{\ell_t}(\theta_t) & 0 & \boldsymbol{\mathcal{R}}_{\ell_t-1}^{\ell_t}(\gamma(\theta_t))
 \end{pmatrix}
 \in \mathbb{C}^{2N_t \times 4N_t}
 \end{align*}
 and the Fourier symbol $\boldsymbol{\mathcal{R}}_{\ell_t-1}^{\ell_t}(\theta_t) \in \mathbb{C}^{N_t \times N_t}$ as defined in Lemma \ref{FourierSymbPandRtime}.
 \end{lemma}
 
 \begin{lemma}[Fourier symbol of prolongation, semi-coarsening]\label{FourierSymbPsemi}
  The following mapping property holds for the prolongation operator $(\underline{\mathbf{P}}_{\ell}^{\ell-1})^s$:
  \begin{align*}
   (\underline{\mathbf{P}}_{\ell}^{\ell-1})^s: \Psi_{\ell_x,\ell_t-1}(\theta_x,2\theta_t) &\to \mathcal{E}_{\ell_x,\ell_t}(\theta_x,\theta_t),\\
   \begin{pmatrix}
    \mathbf{U}_1\\
    \mathbf{U}_2
   \end{pmatrix}
  &\mapsto
 (\widetilde{\boldsymbol{\mathcal{P}}}_{\ell}^{\ell-1})^s(\theta_t)
   \begin{pmatrix}
    \mathbf{U}_1\\
    \mathbf{U}_2
   \end{pmatrix},
  \end{align*}
  with
  \begin{align*}
  (\widetilde{\boldsymbol{\mathcal{P}}}_{\ell}^{\ell-1})^s(\theta_t) :=
   \begin{pmatrix}
     \boldsymbol{\mathcal{P}}_{\ell_t}^{\ell_t-1}(\theta_t) & 0\\
     0 & \boldsymbol{\mathcal{P}}_{\ell_t}^{\ell_t-1}(\theta_t)\\
     \boldsymbol{\mathcal{P}}_{\ell_t}^{\ell_t-1}(\gamma(\theta_t)) & 0\\
     0 & \boldsymbol{\mathcal{P}}_{\ell_t}^{\ell_t-1}(\gamma(\theta_t))
    \end{pmatrix}
   \in \mathbb{C}^{4N_t \times 2 N_t}
  \end{align*}
 and the Fourier symbol $\boldsymbol{\mathcal{P}}_{\ell_t}^{\ell_t-1}(\theta_t)\in \mathbb{C}^{N_t \times N_t}$ as defined in Lemma \ref{FourierSymbPandRtime}.
 \end{lemma}
 
 Analogously to the semi-coarsening case we can get the mapping properties for the full-restriction and full-prolongation operators.
 
 \begin{lemma}[Fourier symbol of restriction, full-coarsening]\label{FourierSymbRFull}
 The following mapping property holds for the restriction operator $(\underline{\mathbf{R}}_{\ell-1}^{\ell})^f$:
   \begin{align*}
   (\underline{\mathbf{R}}_{\ell-1}^{\ell})^f: \mathcal{E}_{\ell_x,\ell_t}(\theta_x,\theta_t) &\to \Psi_{\ell_x-1,\ell_t-1}(2\theta_x,2\theta_t),\\
   \begin{pmatrix}
    \mathbf{U}_1\\
    \mathbf{U}_2\\
    \mathbf{U}_3\\
    \mathbf{U}_4
   \end{pmatrix}
  &\mapsto
  (\widetilde{\boldsymbol{\mathcal{R}}}_{\ell-1}^{\ell})^f(\theta_x,\theta_t)
   \begin{pmatrix}
    \mathbf{U}_1\\
    \mathbf{U}_2\\
    \mathbf{U}_3\\
    \mathbf{U}_4
   \end{pmatrix},
  \end{align*}
  with
 \begin{align*}
  &(\widetilde{\boldsymbol{\mathcal{R}}}_{\ell-1}^{\ell})^f(\theta_x,\theta_t) := \\
 &\begin{pmatrix}
  \hat{\boldsymbol{\mathcal{R}}}_{\ell-1}^{\ell}(\theta_x,\theta_t) & \hat{\boldsymbol{\mathcal{R}}}_{\ell-1}^{\ell}(\gamma(\theta_x),\theta_t) & \hat{\boldsymbol{\mathcal{R}}}_{\ell_t}^{\ell}(\theta_x,\gamma(\theta_t)) & \hat{\boldsymbol{\mathcal{R}}}_{\ell-1}^{\ell}(\gamma(\theta_x),\gamma(\theta_t))
 \end{pmatrix}
 \in \mathbb{C}^{N_t \times 4N_t},
 \end{align*}
 and the Fourier symbol 
 \begin{align*}
  \hat{\boldsymbol{\mathcal{R}}}_{\ell-1}^{\ell}(\theta_x,\theta_t) := \mathcal{R}_{\ell_x-1}^{\ell_x}(\theta_x)\boldsymbol{\mathcal{R}}_{\ell_t-1}^{\ell_t}(\theta_t)\in \mathbb{C}^{N_t \times N_t},
 \end{align*}
 with $\mathcal{R}_{\ell_x-1}^{\ell_x}(\theta_x) \in \mathbb{C}$ from Lemma \ref{FourierSymbPandRspace} and $\boldsymbol{\mathcal{R}}_{\ell_t-1}^{\ell_t}(\theta_t) \in \mathbb{C}^{N_t \times N_t}$ from Lemma \ref{FourierSymbPandRtime}.
 \end{lemma}
 
 \begin{lemma}[Fourier symbol of prolongation, full-coarsening]\label{FourierSymbPfull}
 The following mapping property holds for the prolongation operator $(\underline{\mathbf{P}}_{\ell}^{\ell-1})^f$:
  \begin{align*}
  (\underline{\mathbf{P}}_{\ell}^{\ell-1})^f: \Psi_{\ell_x-1,\ell_t-1}(2\theta_x,2\theta_t) &\to \mathcal{E}_{\ell_x,\ell_t}(\theta_x,\theta_t),\\
  \mathbf{U} & \mapsto
  (\widetilde{\boldsymbol{\mathcal{P}}}_{\ell}^{\ell-1})^f(\theta_t,\theta_x) \mathbf{U},
 \end{align*}
 with
 \begin{align*}
  (\widetilde{\boldsymbol{\mathcal{P}}}_{\ell}^{\ell-1})^f(\theta_t,\theta_x) :=
 \begin{pmatrix}
   \hat{\boldsymbol{\mathcal{P}}}_{\ell}^{\ell-1}(\theta_x,\theta_t)\\
   \hat{\boldsymbol{\mathcal{P}}}_{\ell}^{\ell-1}(\gamma(\theta_x),\theta_t)\\
   \hat{\boldsymbol{\mathcal{P}}}_{\ell}^{\ell-1}(\theta_x,\gamma(\theta_t))\\
   \hat{\boldsymbol{\mathcal{P}}}_{\ell}^{\ell-1}(\gamma(\theta_x),\gamma(\theta_t))
  \end{pmatrix}
  \in \mathbb{C}^{4 N_t \times N_t},
 \end{align*}
 and the Fourier symbol
 \begin{align*}
  \hat{\boldsymbol{\mathcal{P}}}_{\ell}^{\ell-1}(\theta_x,\theta_t):=\mathcal{P}_{\ell_x}^{\ell_x-1}(\theta_x)\boldsymbol{\mathcal{P}}_{\ell_t}^{\ell_t-1}(\theta_t)\in \mathbb{C}^{N_t \times N_t},
 \end{align*}
 with $\mathcal{P}_{\ell_x}^{\ell_x-1} \in \mathbb{C}$ from Lemma \ref{FourierSymbPandRspace} and $\boldsymbol{\mathcal{P}}_{\ell_t}^{\ell_t-1} \in \mathbb{C}^{N_t \times N_t}$ from Lemma \ref{FourierSymbPandRtime}.
 \end{lemma}
  
 With Lemma \ref{FourierSymbolL} we obtain the mapping property for coarse grid correction when semi-coarsening in time is applied:
 \begin{align}\label{MappingSystemMatrixSemi}
 \begin{split}
 (\underline{\mathbf{L}}_{2\tau_{\ell},\xi_{\ell}})^{-1}: \Psi_{\ell_x,\ell_t-1}(\theta_x,2\theta_t) & \rightarrow \Psi_{\ell_x,\ell_t-1}(\theta_x,2\theta_t),\\
 \begin{pmatrix}
 \mathbf{U}_1\\
 \mathbf{U}_2
 \end{pmatrix}
 & \mapsto (\hat{\boldsymbol{\mathcal{L}}}^s_{2\tau_{\ell},\xi_{\ell}}(\theta_x,2\theta_t))^{-1}
 \begin{pmatrix}
 \mathbf{U}_1\\
 \mathbf{U}_2
 \end{pmatrix},
 \end{split}
 \end{align}
 with
 \begin{align*}
 (\hat{\boldsymbol{\mathcal{L}}}^s_{2\tau_{\ell},\xi_{\ell}}(\theta_x,2\theta_t))^{-1} :=
 \begin{pmatrix}
 (\boldsymbol{\mathcal{L}}_{2\tau_{\ell},\xi_{\ell}}(\theta_x,2\theta_t))^{-1} & 0 \\
 0 & (\boldsymbol{\mathcal{L}}_{2\tau_{\ell},\xi_{\ell}}(\gamma(\theta_x),2\theta_t))^{-1} 
 \end{pmatrix}
 \in \mathbb{C}^{2N_t\times2N_t}.
 \end{align*}
 
 A complication arises for frequencies $(\theta_x,\theta_t)$ such that $\boldsymbol{\mathcal{L}}_{2\tau_{\ell},\xi_{\ell}}(\theta_x,2\theta_t)=0$. For some more discussion of the reasons for this formal complication we refer to \cite{Trottenberg2001}. In order to make sure that $\hat{\boldsymbol{\mathcal{L}}}^s$ exists, we exclude the set of frequencies
 \begin{align}\label{SetSemi}
  \Lambda_s := \left\{ (\theta_x,\theta_t) \in \left( -\pi,\pi \right]\times \left( -\frac{\pi}{2},\frac{\pi}{2} \right]: \boldsymbol{\mathcal{L}}_{\tau_{\ell},\xi_{\ell}}(\theta_x,\theta_t)=0 \text{ or } \boldsymbol{\mathcal{L}}_{2\tau_{\ell},\xi_{\ell}}(\theta_x,2\theta_t) = 0 \right\}.
 \end{align}
 
 For the full-coarsening case we obtain the mapping property
 \begin{align}\label{MappingSystemMatrixFull}
 \begin{split}
 (\underline{\mathbf{L}}_{2\tau_{\ell},2\xi_{\ell}})^{-1}: \Psi_{\ell_x-1,\ell_t-1}(2\theta_x,2\theta_t) & \rightarrow \Psi_{\ell_x-1,\ell_t-1}(2\theta_x,2\theta_t),\\
 \mathbf{U}
 & \mapsto (\hat{\boldsymbol{\mathcal{L}}}^f_{2\tau_{\ell},2\xi_{\ell}}(2\theta_x,2\theta_t))^{-1} \mathbf{U},
 \end{split}
 \end{align}
 with
 \begin{align}\label{SetFull}
 (\hat{\boldsymbol{\mathcal{L}}}^f_{2\tau_{\ell},2\xi_{\ell}}(2\theta_x,2\theta_t))^{-1} :=
 (\boldsymbol{\mathcal{L}}_{2\tau_{\ell},2\xi_{\ell}}(2\theta_x,2\theta_t))^{-1}
 \in \mathbb{C}^{N_t\times N_t}.
 \end{align}
 
 As before, a complication arises for frequencies $(\theta_x,\theta_t)$ such that $\boldsymbol{\mathcal{L}}_{2\tau_{\ell},2\xi_{\ell}}(2\theta_x,2\theta_t)=0$. In order to make sure that $\hat{\boldsymbol{\mathcal{L}}}^f$ exists, we exclude the set of frequencies
 \begin{align*}
  \Lambda_f := \left\{ (\theta_x,\theta_t) \in \left( -\frac{\pi}{2},\frac{\pi}{2} \right]^2: \boldsymbol{\mathcal{L}}_{\tau_{\ell},\xi_{\ell}}(\theta_x,\theta_t)=0 \text{ or } \boldsymbol{\mathcal{L}}_{2\tau_{\ell},2\xi_{\ell}}(2\theta_x,2\theta_t) = 0 \right\}.
 \end{align*}
 
 We are now able to get the Fourier symbol of the two-grid operators and calculate the asymptotic convergence factors.
 
 \begin{theorem}[Fourier symbol of two-grid operator, semi-coarsening]\label{FourierSymbTwoGridsemi}
 For $(\theta_x,\theta_t)\in \Theta_{\ell_x,\ell_t}^{low,f}$ the following mapping property holds for the two-grid operator $\underline{\mathbf{M}}^s_{\tau_{\ell},\xi_{\ell}}$ in \eqref{Semi2GridMatrix} with semi-coarsening in time:
 \begin{align*}
 \underline{\mathbf{M}}^s_{\tau_{\ell},\xi_{\ell}}: \mathcal{E}_{\ell_x,\ell_t}(\theta_x,\theta_t) & \rightarrow \mathcal{E}_{\ell_x,\ell_t}(\theta_x,\theta_t),\\
 \begin{pmatrix}
 \mathbf{U}_1\\
 \mathbf{U}_2\\
 \mathbf{U}_3\\
 \mathbf{U}_4
 \end{pmatrix}
 & \mapsto
 \boldsymbol{\mathcal{M}}^s(\theta_x,\theta_t)
 \begin{pmatrix}
 \mathbf{U}_1\\
 \mathbf{U}_2\\
 \mathbf{U}_3\\
 \mathbf{U}_4
 \end{pmatrix},
 \end{align*}
 with
 \begin{align}\label{SymbolSemi}
 \begin{split}
 \boldsymbol{\mathcal{M}}^s(\theta_x,\theta_t) &:= \widetilde{\boldsymbol{\mathcal{S}}}_{\tau_{\ell},\xi_{\ell}}^{\nu_2}(\theta_x,\theta_t)(\mathbf{I}_{4N_t}-(\widetilde{\boldsymbol{\mathcal{P}}}^{\ell-1}_{\ell})^{s}(\theta_t)(\hat{\boldsymbol{\mathcal{L}}}^s_{2\tau_{\ell},\xi_{\ell}}(\theta_x,2\theta_t))^{-1}\\
 & ~~~~ (\widetilde{\boldsymbol{\mathcal{R}}}_{\ell-1}^{\ell})^{s}(\theta_t) \widetilde{\boldsymbol{\mathcal{L}}}_{\tau_{\ell},\xi_{\ell}}(\theta_x,\theta_t)) \widetilde{\boldsymbol{\mathcal{S}}}_{\tau_{\ell},\xi_{\ell}}^{\nu_1}(\theta_x,\theta_t) \in \mathbb{C}^{4N_t \times 4N_t}.
 \end{split}
 \end{align}
 \end{theorem}
 
 \begin{proof}
 The two-grid operator for semi-coarsening is given by
 \begin{align*}
 \underline{\mathbf{M}}^s_{\tau_{\ell},\xi_{\ell}} = \underline{\mathbf{S}}_{\tau_{\ell},\xi_{\ell}}^{\nu_2} (\underline{\mathbf{I}} - (\underline{\mathbf{P}}_{\ell}^{\ell-1})_s (\underline{\mathbf{L}}_{2\tau_{\ell},\xi_{\ell}})^{-1} (\underline{\mathbf{R}}_{\ell-1}^{\ell})_s \underline{\mathbf{L}}_{\tau_{\ell},\xi_{\ell}})\underline{\mathbf{S}}_{\tau_{\ell},\xi_{\ell}}^{\nu_1}.
 \end{align*}
 By previous results we obtain
 \begin{align*}
 \underline{\mathbf{M}}^s_{\tau_{\ell},\xi_{\ell}}: & \mathcal{E}_{\ell_x,\ell_t} \xrightarrow[]{\eqref{MappingSmoother}} \mathcal{E}_{\ell_x,\ell_t} \xrightarrow[]{\eqref{MappingSystemMatrix}} \mathcal{E}_{\ell_x,\ell_t} \xrightarrow[]{\ref{FourierSymbRsemi}} \Psi_{\ell_x,\ell_t-1}(\theta_x,2\theta_t) \\
 & \xrightarrow[]{\eqref{MappingSystemMatrixSemi}} \Psi_{\ell_x,\ell_t-1}(\theta_x,2\theta_t) \xrightarrow[]{\ref{FourierSymbPsemi}} \mathcal{E}_{\ell_x,\ell_t}
 \xrightarrow[]{\eqref{MappingSmoother}} \mathcal{E}_{\ell_x,\ell_t},
 \end{align*}
 with the mapping
 \begingroup
 \allowdisplaybreaks
 \begin{align*}
 \begin{pmatrix}
 \mathbf{U}_1\\
 \mathbf{U}_2\\
 \mathbf{U}_3\\
 \mathbf{U}_4
 \end{pmatrix} 
 & \mapsto
 \widetilde{\boldsymbol{\mathcal{S}}}_{\tau_{\ell},\xi_{\ell}}^{\nu_1}
 \begin{pmatrix}
 \mathbf{U}_1\\
 \mathbf{U}_2\\
 \mathbf{U}_3\\
 \mathbf{U}_4
 \end{pmatrix}
 \mapsto
 \widetilde{\boldsymbol{\mathcal{L}}}_{\tau_{\ell},\xi_{\ell}} \widetilde{\boldsymbol{\mathcal{S}}}_{\tau_{\ell},\xi_{\ell}}^{\nu_1}
 \begin{pmatrix}
 \mathbf{U}_1\\
 \mathbf{U}_2\\
 \mathbf{U}_3\\
 \mathbf{U}_4
 \end{pmatrix}
 \mapsto
 (\widetilde{\boldsymbol{\mathcal{R}}}_{\ell-1}^{\ell})^{s} \widetilde{\boldsymbol{\mathcal{L}}}_{\tau_{\ell},\xi_{\ell}} \widetilde{\boldsymbol{\mathcal{S}}}_{\tau_{\ell},\xi_{\ell}}^{\nu_1}
 \begin{pmatrix}
 \mathbf{U}_1\\
 \mathbf{U}_2\\
 \mathbf{U}_3\\
 \mathbf{U}_4
 \end{pmatrix}\\
 &\mapsto
 (\hat{\boldsymbol{\mathcal{L}}}^s_{2\tau_{\ell},\xi_{\ell}})^{-1} (\widetilde{\boldsymbol{\mathcal{R}}}_{\ell-1}^{\ell})^{s} \widetilde{\boldsymbol{\mathcal{L}}}_{\tau_{\ell},\xi_{\ell}} \widetilde{\boldsymbol{\mathcal{S}}}_{\tau_{\ell},\xi_{\ell}}^{\nu_1}
 \begin{pmatrix}
 \mathbf{U}_1\\
 \mathbf{U}_2\\
 \mathbf{U}_3\\
 \mathbf{U}_4
 \end{pmatrix}\\
 &\mapsto
 (\widetilde{\boldsymbol{\mathcal{P}}}^{\ell-1}_{\ell})^{s}(\hat{\boldsymbol{\mathcal{L}}}^s_{2\tau_{\ell},\xi_{\ell}})^{-1} (\widetilde{\boldsymbol{\mathcal{R}}}_{\ell-1}^{\ell})^{s} \widetilde{\boldsymbol{\mathcal{L}}}_{\tau_{\ell},\xi_{\ell}} \widetilde{\boldsymbol{\mathcal{S}}}_{\tau_{\ell},\xi_{\ell}}^{\nu_1}
 \begin{pmatrix}
 \mathbf{U}_1\\
 \mathbf{U}_2\\
 \mathbf{U}_3\\
 \mathbf{U}_4
 \end{pmatrix}\\
 &\mapsto
 (\mathbf{I}_{4N_t}-(\widetilde{\boldsymbol{\mathcal{P}}}^{\ell-1}_{\ell})^{s}(\hat{\boldsymbol{\mathcal{L}}}^s_{2\tau_{\ell},\xi_{\ell}})^{-1} (\widetilde{\boldsymbol{\mathcal{R}}}_{\ell-1}^{\ell})^{s} \widetilde{\boldsymbol{\mathcal{L}}}_{\tau_{\ell},\xi_{\ell}}) \widetilde{\boldsymbol{\mathcal{S}}}_{\tau_{\ell},\xi_{\ell}}^{\nu_1}
 \begin{pmatrix}
 \mathbf{U}_1\\
 \mathbf{U}_2\\
 \mathbf{U}_3\\
 \mathbf{U}_4
 \end{pmatrix}\\
 &\mapsto
 \widetilde{\boldsymbol{\mathcal{S}}}_{\tau_{\ell},\xi_{\ell}}^{\nu_2}(\mathbf{I}_{4N_t}-(\widetilde{\boldsymbol{\mathcal{P}}}^{\ell-1}_{\ell})^{s}(\hat{\boldsymbol{\mathcal{L}}}^s_{2\tau_{\ell},\xi_{\ell}})^{-1} (\widetilde{\boldsymbol{\mathcal{R}}}_{\ell-1}^{\ell})^{s} \widetilde{\boldsymbol{\mathcal{L}}}_{\tau_{\ell},\xi_{\ell}}) \widetilde{\boldsymbol{\mathcal{S}}}_{\tau_{\ell},\xi_{\ell}}^{\nu_1}
 \begin{pmatrix}
 \mathbf{U}_1\\
 \mathbf{U}_2\\
 \mathbf{U}_3\\
 \mathbf{U}_4
 \end{pmatrix}.
 \end{align*}
 \endgroup
 \end{proof}
 
 \begin{theorem}[Fourier symbol of two-grid operator, full-coarsening]\label{FourierSymbTwoGridfull}
 For $(\theta_x,\theta_t) \in \Theta_{{\ell}_x,{\ell}_t}^{low,f}$ the following mapping property holds for the two-grid operator $\underline{\mathbf{M}}^f_{\tau_{\ell},\xi_{\ell}}$ in \eqref{Full2GridMatrix} with space-time coarsening:
 \begin{align*}
 \underline{\mathbf{M}}^f_{\tau_{\ell},\xi_{\ell}}: \mathcal{E}_{{\ell}_x,{\ell}_t}(\theta_x,\theta_t) & \rightarrow \mathcal{E}_{{\ell}_x,{\ell}_t}(\theta_x,\theta_t),\\
 \begin{pmatrix}
 \mathbf{U}_1\\
 \mathbf{U}_2\\
 \mathbf{U}_3\\
 \mathbf{U}_4
 \end{pmatrix}
 & \mapsto
 \boldsymbol{\mathcal{M}}^f(\theta_x,\theta_t)
 \begin{pmatrix}
 \mathbf{U}_1\\
 \mathbf{U}_2\\
 \mathbf{U}_3\\
 \mathbf{U}_4
 \end{pmatrix},
 \end{align*}
 with
 \begin{align}\label{SymbolFull}
 \begin{split}
 \boldsymbol{\mathcal{M}}^f(\theta_x,\theta_t) &:=  \widetilde{\boldsymbol{\mathcal{S}}}_{\tau_{\ell},\xi_{\ell}}^{\nu_2}(\theta_x,\theta_t)(\mathbf{I}_{4N_t}-(\widetilde{\boldsymbol{\mathcal{P}}}^{\ell-1}_{\ell})^{f}(\theta_x,\theta_t)(\hat{\boldsymbol{\mathcal{L}}}^f_{2\tau_{\ell},2\xi_{\ell}}(2\theta_x,2\theta_t))^{-1}\\ 
 & ~~~~ (\widetilde{\boldsymbol{\mathcal{R}}}_{\ell-1}^{\ell})^{f}(\theta_x,\theta_t) \widetilde{\boldsymbol{\mathcal{L}}}_{\tau_{\ell},\xi_{\ell}}(\theta_x,\theta_t)) \widetilde{\boldsymbol{\mathcal{S}}}_{\tau_{\ell},\xi_{\ell}}^{\nu_1}(\theta_x,\theta_t) \in \mathbb{C}^{4N_t \times 4N_t}.
 \end{split}
 \end{align}
 \end{theorem}
 
 \begin{proof}
 The proof follows analogous to the previous one.
 \end{proof}

 \section{Smoothing Analysis}\label{SectionSmoothingAnalysis}

 We now have all tools at hand to analyze the elements of the multigrid iteration. We start with the smoother. The asymptotic smoothing factor of the damped block Jacobi method \eqref{ItMatrixDampedJacobi} can be measured by computing the spectral radius of its symbol $\boldsymbol{\mathcal{S}}_{\tau_{\ell},\xi_{\ell}}(\theta_x,\theta_t)$, which is of much smaller size and thus makes the calculations feasible.
 
 \begin{definition}[\cite{Wesseling2004}]\label{AsymptSmoothingFactors}
 We define the \textit{asymptotic smoothing factors} for semi- and full-coarsening as
 \begin{align*}
 \varrho(\boldsymbol{\mathcal{S}}^s) &:= \max \{ \varrho (\boldsymbol{\mathcal{S}}_{\tau_{\ell},\xi_{\ell}}(\theta_x,\theta_t)) : (\theta_x,\theta_t) \in \Theta_{{\ell}_x,{\ell}_t}^{high,s}\},\\
 \varrho(\boldsymbol{\mathcal{S}}^f) &:= \max \{ \varrho( \boldsymbol{\mathcal{S}}_{\tau_{\ell},\xi_{\ell}}(\theta_x,\theta_t) ): (\theta_x,\theta_t) \in \Theta_{{\ell}_x,{\ell}_t}^{high,f}\},
 \end{align*}
 with $\boldsymbol{\mathcal{S}}_{\tau_{\ell},\xi_{\ell}}(\theta_x,\theta_t)$ the Fourier symbol of the smoother and the set of frequencies defined in \eqref{FrequenciesSemi} and \eqref{FrequenciesFull}.
 \end{definition}

 \begin{lemma}\label{SpectralRadiusS}
 The spectral radius of the Fourier symbol of the smoother $\boldsymbol{\mathcal{S}}_{\tau_{\ell},\xi_{\ell}}(\theta_x,\theta_t)$ is given by
 \begin{align*}
  \rho(\boldsymbol{\mathcal{S}}_{\tau_{\ell},\xi_{\ell}}(\theta_x,\theta_t)) = \max \{|1-\omega_t|,S(\omega_t,\theta_x,\theta_t)\}
 \end{align*}
 with
 \begin{align}\label{FunctionSpectralRadiusS}
  S(\omega_t,\theta_x,\theta_t) := |1-\omega_t+e^{-\mathrm{i}\theta_t}\omega_t R(-\mu\beta(\theta_x))|,
 \end{align}
 $R$ the stability function of the DG-SEM time stepping scheme, $\beta(\theta_x):= 1-e^{-\mathrm{i}\theta_x}$ and CFL number $\mu := \frac{a \Delta \tau_{\ell}}{\Delta x_{\ell}}$.
 \end{lemma}
  
 \begin{proof}
 The eigenvalues of $\boldsymbol{\mathcal{S}}_{\tau_{\ell},\xi_{\ell}}(\theta_x,\theta_t)$ are given by
 \begin{align*}
 \sigma(\boldsymbol{\mathcal{S}}_{\tau_{\ell},\xi_{\ell}}(\theta_x,\theta_t)) & = 1-\omega_t+e^{-\mathrm{i}\theta_t}\omega_t\sigma\left( \left(\mathbf{K}_{\tau_{\ell}}+\frac{a}{\Delta x}(-e^{-\mathrm{i}\theta_x}+1)\mathbf{M}_{\tau_{\ell}}\right)^{-1}\mathbf{C}_{\tau_{\ell}} \right).
 \end{align*}
 With Lemma \ref{Eigenvalues} we can compute the spectrum as 
 \begin{align*}
 \sigma(\boldsymbol{\mathcal{S}}_{\tau_{\ell},\xi_{\ell}}(\theta_x,\theta_t)) = \{ 1-\omega_t,1-\omega_t+e^{-\mathrm{i}\theta_t}\omega_t R(-\mu\beta(\theta_x)) \}.
 \end{align*}
 Therefore it follows that
 \begin{align*}
 \rho(\boldsymbol{\mathcal{S}}_{\tau_{\ell},\xi_{\ell}}(\theta_x,\theta_t)) = \max \{ |1-\omega_t|,|1-\omega_t+e^{-\mathrm{i}\theta_k}\omega_t R(-\mu\beta(\theta_x))| \}.
 \end{align*}
 \end{proof}

 \subsection{Optimal Damping Parameter}
 
 The goal is to find a smoother with optimal smoothing behavior, i.e. to find a damping parameter $\omega_t$ for the block Jacobi smoother such that the high frequencies frequencies are smoothed as efficiently as possible. We analyze the Fourier symbol $\boldsymbol{\mathcal{S}}_{\tau_{\ell},\xi_{\ell}}$ to find the optimal damping parameter $\omega_t \in (0,1]$ in \eqref{ItMatrixDampedJacobi}. In order to do so, the frequencies which are damped less efficiently need to be determined. These are also called worst case frequencies:
 
 \begin{definition}
  The \textit{worst case frequencies} for the Fourier symbol $\boldsymbol{\mathcal{S}}_{\tau_{\ell},\xi_{\ell}}$ of the smoother are defined as those high frequencies that are damped least efficiently.
 \end{definition}
 
 This can be done by calculating the maximum of $|1-\omega_t|$ and
 \begin{align}\label{DefWorstCaseFreq}
    (\theta_x^*(\omega_t,\mu),\theta_t^*(\omega_t,\mu)) := \argsup_{\substack{(\theta_x,\theta_t) \in \Theta^{high}}} S(\omega_t,\theta_x,\theta_t),
 \end{align}
 with the function $S$ defined in Lemma \ref{SpectralRadiusS}, see equation \eqref{FunctionSpectralRadiusS}.
 
 Straightforward calculations give 
 \begin{align*}
  S(\omega_t,\theta_x,\theta_t)^2 & = |1-\omega_t+e^{-\mathrm{i}\theta_t}\omega_t R(-\mu\beta(\theta_x))|^2 \\
  & =  (1-\omega_t)^2 + 2\omega_t (1-\omega_t)(\cos(\theta_t) \text{Re}(R(-\mu\beta(\theta_x))) + \sin(\theta_t) \text{Im}(R(-\mu\beta(\theta_x)))) \\
  &~~ + \omega_t^2 |R(-\mu\beta(\theta_x))|^2.
 \end{align*}
 
 When considering implicit time integration, large CFL numbers $\mu \gg 0$ are of interest. Then
 \begin{align*}
  \cos(\theta_t) \text{Re}(R(-\mu\beta(\theta_x))) + \sin(\theta_t) \text{Im}(R(-\mu\beta(\theta_x))) \xrightarrow[\mu\to\infty]{} 0,
 \end{align*}
 since $\cos(\theta_t), \sin(\theta_t) \in [-1,1]$, and the method is L-stable, see Corollary \ref{stabilityFunctionDG}, thus $R(-z) \rightarrow 0$ for $z \rightarrow \infty$ and $\text{Re}(-\mu\beta(\theta_x)) \le 0$ for $\theta_x \in [-\pi,\pi]$.
 
 To find the worst case frequencies in space for large CFL numbers $\mu$ we thus need to maximize $|R(-\mu\beta(\theta_x))|^2$. From L-stability it follows that $|R(-\mu\beta(\theta_x))|^2 \le 1$ and moreover we have by definition of the Pad\'e approximant that $R(0)=1$. Thus we have found the worst case frequency in space:
 \begin{align*}
 \theta_x^* = \argsup_{\substack{\theta_x \in (-\pi,\pi]}} S(\omega_t,\theta_x,\theta_t) = 0.
 \end{align*} 
 Evaluating \eqref{FunctionSpectralRadiusS} at $\theta_x = \theta_x^*=0$ gives
 \begin{align*}
 S(\omega_t,0,\theta_t) & = |1-\omega_t+e^{-\mathrm{i}\theta_t}\omega_t R(-\mu\beta(\theta_x))|\\
 & = \sqrt{(1-\omega_t)^2 + 2\omega_t (1-\omega_t) \cos(\theta_t) + \omega_t^2}.
 \end{align*} 
 With this, it follows that the worst case frequencies in time are given by
 \begin{align*}
 \theta_t^{*} & = \argsup_{\substack{\theta_t \in [\pi/2,\pi]}} S(\omega_t,\theta_x^*,\theta_t) = \frac{\pi}{2},\\
 \theta_t^{*} & = \argsup_{\substack{\theta_t \in [-\pi,-\pi/2]}} S(\omega_t,\theta_x^*,\theta_t) = -\frac{\pi}{2}.
 \end{align*}
 Thus, we have found worst case frequencies $(\theta_x^*,\theta_t^*) \in \Theta_{\ell_x,\ell_t}^{high,s}$ for the semi-coarsening strategy as well as $(\theta_x^*,\theta_t^*) \in \Theta_{\ell_x,\ell_t}^{high,f}$ for the full-coarsening strategy. The optimal damping parameter can then be calculated by
 \begin{align*}
 \omega_t^* = \arginf_{\omega_t \in (0,1]}S(\omega_t,\theta_x^*,\theta_t^{*}) = 0.5.
 \end{align*}
 
 \subsection{Asymptotic Smoothing Factor}
 
 With the optimal damping parameter $\omega_t^* = 0.5$ and the worst case frequencies $(\theta_x^*,\theta_t^*)$ at hand we can calculate the asymptotic smoothing factor from Definition \ref{AsymptSmoothingFactors}. 
 
 In the case of full space-time coarsening we get for $\mu$ large enough
 \begin{align*}
 \varrho(\boldsymbol{\mathcal{S}}) = \rho(\boldsymbol{\mathcal{S}}_{\tau_{\ell},\xi_{\ell}}(\theta_x^*,\theta_t^*)) = \max \left\{ |0.5|,S(0.5,0,\frac{\pi}{2}) \right\} = \max \{ 0.5,\sqrt{0.5}\} = \frac{1}{\sqrt{2}},
 \end{align*} 
 and for semi-coarsening in time
 \begin{align*}
 \varrho(\boldsymbol{\mathcal{S}}) = \rho(\boldsymbol{\mathcal{S}}_{\tau_{\ell},\xi_{\ell}}(\theta_x^*,\theta_t^*)) = \max \left\{ |0.5|,S(0.5,0,-\frac{\pi}{2}) \right\} = \max \{0.5,\sqrt{0.5}\} = \frac{1}{\sqrt{2}}.
 \end{align*}

 \section{Two-Grid Analysis}\label{SectionTwoGridAnalysis}
 
 In this section we analyze the two-grid iteration for full- and semi-coarsening by studying the corresponding iteration matrices $\underline{\mathbf{M}}_{\tau_{\ell},\xi_{\ell}}^f$ and $\underline{\mathbf{M}}_{\tau_{\ell},\xi_{\ell}}^s$, see equations \eqref{Full2GridMatrix} and \eqref{Semi2GridMatrix}. With Theorems \ref{FourierSymbTwoGridsemi} and \ref{FourierSymbTwoGridfull} we can analyze the asymptotic convergence behavior of the two-grid cycle by computing the maximal spectral radius of the Fourier symbols $\boldsymbol{\mathcal{M}}_{\mu}^s(\theta_x,\theta_t)$ and $\boldsymbol{\mathcal{M}}_{\mu}^f(\theta_x,\theta_t)$ for $(\theta_x,\theta_t) \in \Theta_{{\ell}_x,{\ell}_t}^{low,f}$, see \eqref{SymbolFull} and \eqref{SymbolSemi}.
 
 \begin{definition}\label{AsymptConvFactors}
 We define the \textit{asymptotic two-grid convergence factors} as
 \begin{align*}
 \varrho(\boldsymbol{\mathcal{M}}^s) &:= \max \{ \varrho (\boldsymbol{\mathcal{M}}^s(\theta_x,\theta_t)) : (\theta_x,\theta_t) \in \Theta_{{\ell}_x,{\ell}_t}^{low,f} \setminus \Lambda_s \},\\
 \varrho(\boldsymbol{\mathcal{M}}^f) &:= \max \{ \varrho( \boldsymbol{\mathcal{M}}^f(\theta_x,\theta_t) ): (\theta_x,\theta_t) \in \Theta_{{\ell}_x,{\ell}_t}^{low,f} \setminus \Lambda_f\},
 \end{align*}
 with $\boldsymbol{\mathcal{M}}^s(\theta_x,\theta_t)$ and $\boldsymbol{\mathcal{M}}^f(\theta_x,\theta_t)$ the symbols of the two-grid iteration matrices and $\Lambda_s$ defined in \eqref{SetSemi} and $\Lambda_f$ defined in \eqref{SetFull}.
 \end{definition}
 
 To derive $\varrho(\boldsymbol{\mathcal{M}}^s)$ and $\varrho(\boldsymbol{\mathcal{M}}^f)$ for a given CFL number $\mu \in \mathbb{R}_{+}$ and a given polynomial degree $p_t \in \mathbb{N}$, it is necessary to compute the eigenvalues of
 \begin{align*}
 \boldsymbol{\mathcal{M}}^s(\theta_x,\theta_t) \in \mathbb{C}^{4N_t \times 4N_t} ~~~~\text{and}~~~~ \boldsymbol{\mathcal{M}}^f(\theta_x,\theta_t) \in \mathbb{C}^{4N_t \times 4N_t},
 \end{align*}
 with $N_t = p_t+1$ for all low frequencies $(\theta_x,\theta_t) \in \Theta_{{\ell}_x,{\ell}_t}^{low,s}$ respectively $(\theta_x,\theta_t) \in \Theta_{{\ell}_x,{\ell}_t}^{low,f}$. 

 It is difficult to find analytical expressions for the eigenvalues of the two-grid operators $\boldsymbol{\mathcal{M}}^f(\theta_x,\theta_t)$ and $\boldsymbol{\mathcal{M}}^s(\theta_x,\theta_t)$ since they are the product of several Fourier symbols which itself are complex functions. We therefore compute the eigenvalues numerically for all frequencies $(\theta_x,\theta_t) \in \Theta_{{\ell}_x,{\ell}_t}^{low,f}$ and $(\theta_x,\theta_t) \in \Theta_{{\ell}_x,{\ell}_t}^{low,s}$. 
 We consider a space-time discretization with $N_x$ volumes in space and $N$ space-time slabs on the domain $[0,1] \times [0,T]$, where we adapt $T$ via the CFL number $\mu = \frac{\Delta t}{\Delta x} \in [1,800]$.

The results for the LFA can be seen in Figure~\ref{fig:LFA} for $N_x=2^5$, $N=2^3$ to the left and for $N_x=2^{10}$, $N=2^3$ to the right, $\mu \in [1,10,50,100,200,400,600,800]$ for both coarsening strategies, referred to as semi and full, and $p_t=0$ and $p_t=1$, respectively.
 
They show that for high CFL numbers the multigrid solver has excellent asymptotic convergence rates about $0.5$ for $p_t=0$ which can be improved to $0.375$ by increasing the polynomial degree in time to $p_t=1$. Moreover, these convergence rates are independent of the coarsening strategy.

\begin{figure}[ht!]
\includegraphics[width=0.49\textwidth]{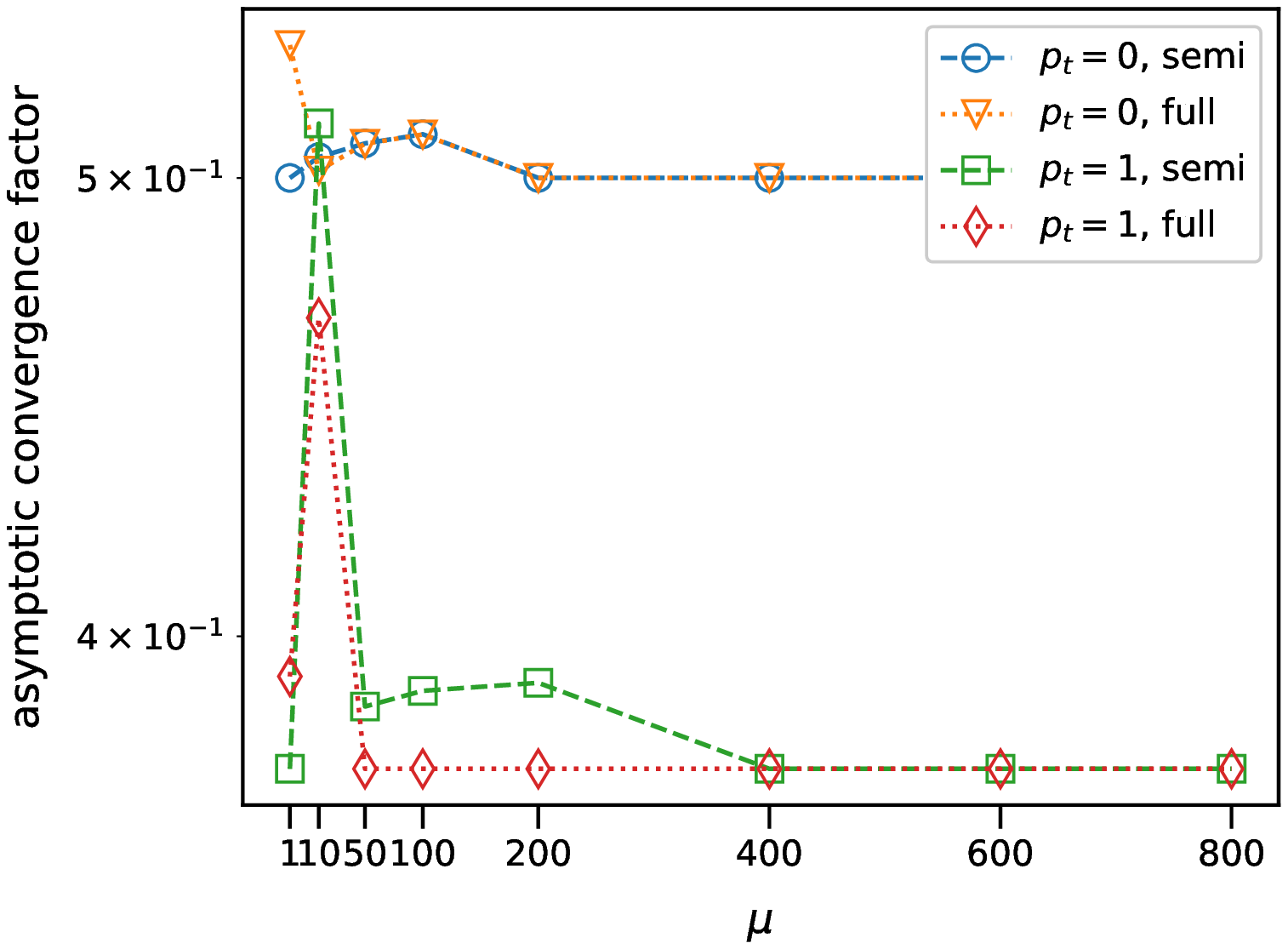}
\includegraphics[width=0.49\textwidth]{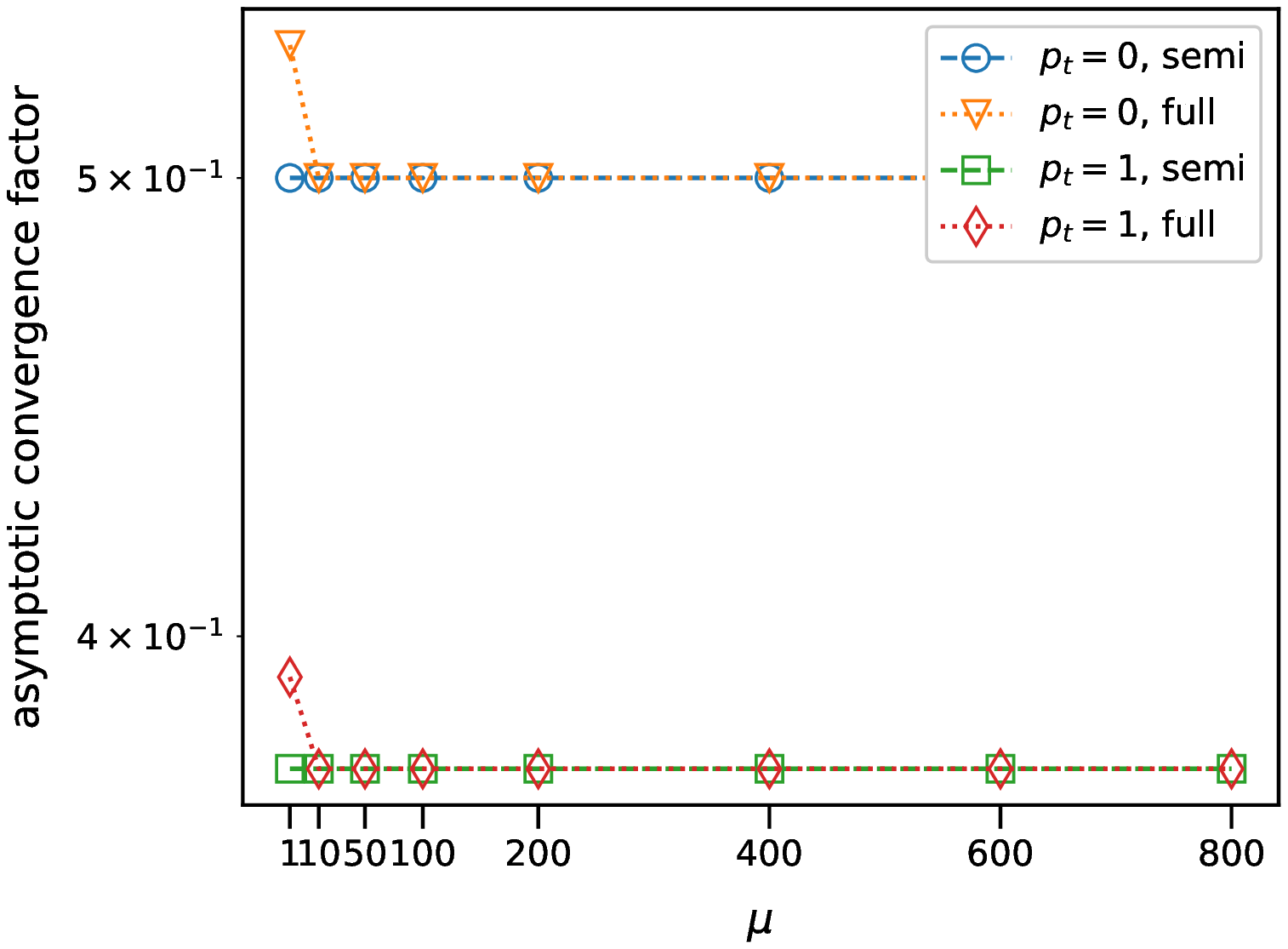}
\caption{Results of the LFA for the test problem: left for $N_x=2^5$, $N=2^3$, right for $N_x=2^{10}$, $N=2^3$.}
\label{fig:LFA}
\end{figure}

\section{Numerical Examples}\label{SectionNumericalResults}

We now solve the system \eqref{MG} using this two-grid method. Periodic boundary conditions in space and time, needed to perform the LFA, cannot be used for the numerical tests since this results in singular iteration matrices. We therefore adjust test case \eqref{advectionProblem} and consider the following problems in one respectively two spatial dimensions:
\begin{align}
 u_t + a u_x = 0,~ a = 1,~ (x,t) \in (0,1] \times (0,T],
\end{align}
with solution $u(x,t) = \sin(\pi(x-t))$, and
\begin{align}
 \mathbf{u}_t + \mathbf{a} \cdot \mathbf{u}_x = \mathbf{0},~ \mathbf{a} = (1,1],~ (\mathbf{x},t) \in (0,1]^2 \times [0,T],
\end{align}
with solution $\mathbf{u}(\mathbf{x},t) = \sin(\pi(x_1-t)) \sin(\pi(x_2-t))$.
Moreover, we consider a full space-time DG-SEM, i.e. a DG-SEM in space and time. As before, the time interval is determined via $T = N \mu \Delta x$.

All numerical tests in this section are performed using the Python interface of DUNE \cite{Dedner2020} on an Intel Xeon E5-2650 v3 processor (Haswell) on the LUNARC Aurora cluster at Lund University.

We calculate the asymptotic convergence rate from 60 multigrid iterations by
\begin{align*}
\max_{i=1,\dots,59} \frac{\| \mathbf{r}^{i+1} \|_2}{\|\mathbf{r}^i\|_2},~ \mathbf{r}^i= \underline{\mathbf{L}}_{\tau,\xi}\underline{\mathbf{u}}^i-\underline{\mathbf{b}}.
\end{align*}
The results for one spatial dimension can be seen in Figure~\ref{fig:Num} to the left, with $N=2^3$ and $N_x=2^{10}$. The convergence rates converge for both coarsening strategies to approximately $0.25$ for $p_t=p_x=0$ and to approximately $0.3$ for $p_t=p_x=1$. The CFL number to achieve these convergence rates increases when increasing the order of the polynomial approximation. Moreover, we get slightly higher convergence rates for small CFL numbers when using the semi-coarsening strategy. 

Increasing the number of spatial dimensions we measure the numerical convergence rates for $N=2^3$ and $N_x=N_y=2^5$. The results can be seen in Figure~\ref{fig:Num} to the right. Here, the convergence rates for both coarsening strategies and different DG orders are very similar, converging to approximately $0.25$. However, we notice some oscillations for the semi-coarsening ansatz with $p_t=p_x=1$. This might vanish when increasing the CFL number. While the numerical convergence rates are similar to the one-dimensional case for $p_t=p_x=0$, they improve slightly for $p_x=p_t=1$ when increasing the number of spatial dimensions.

\begin{figure}[ht!]
\includegraphics[width=0.49\textwidth]{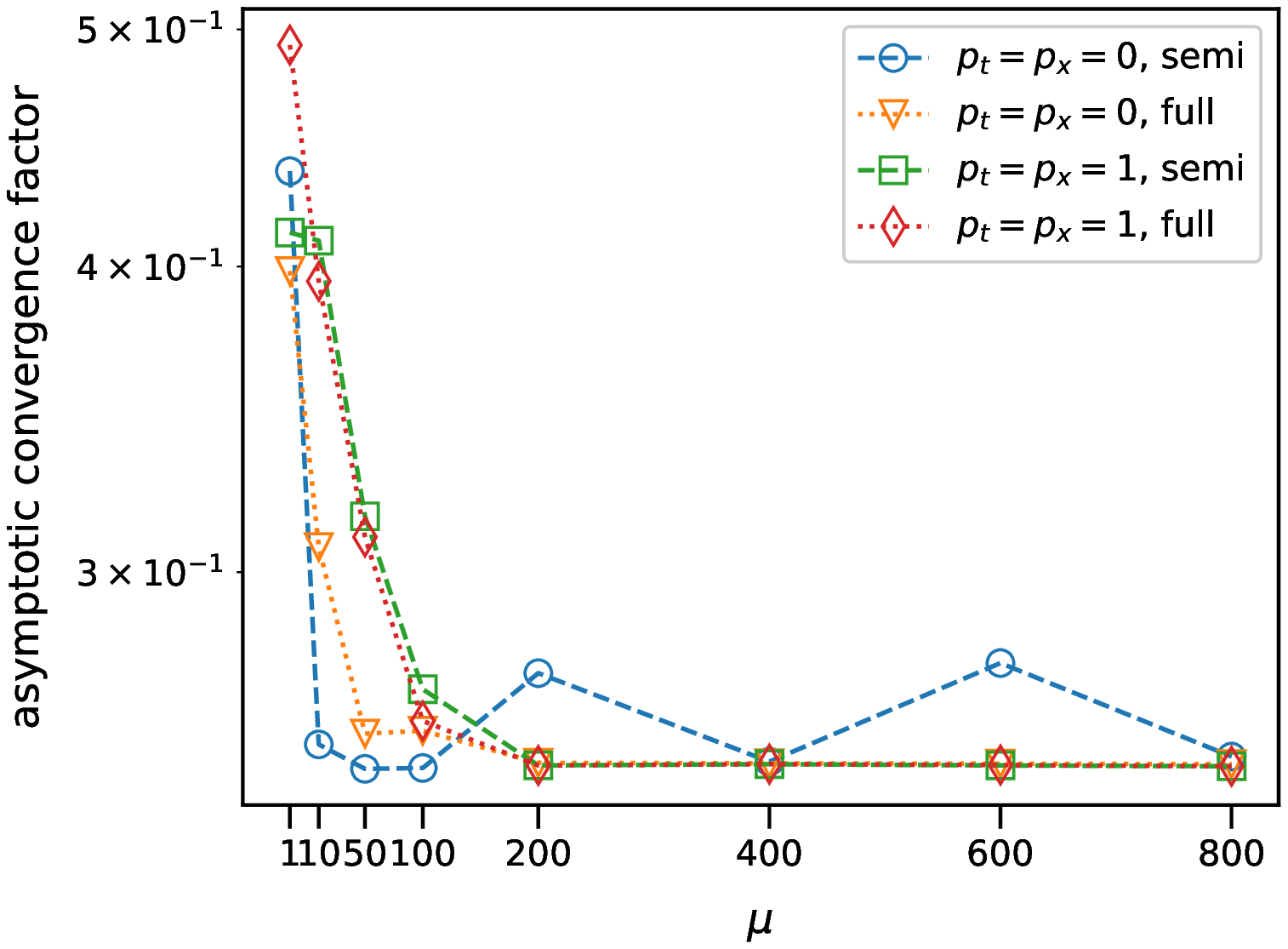}
\includegraphics[width=0.49\textwidth]{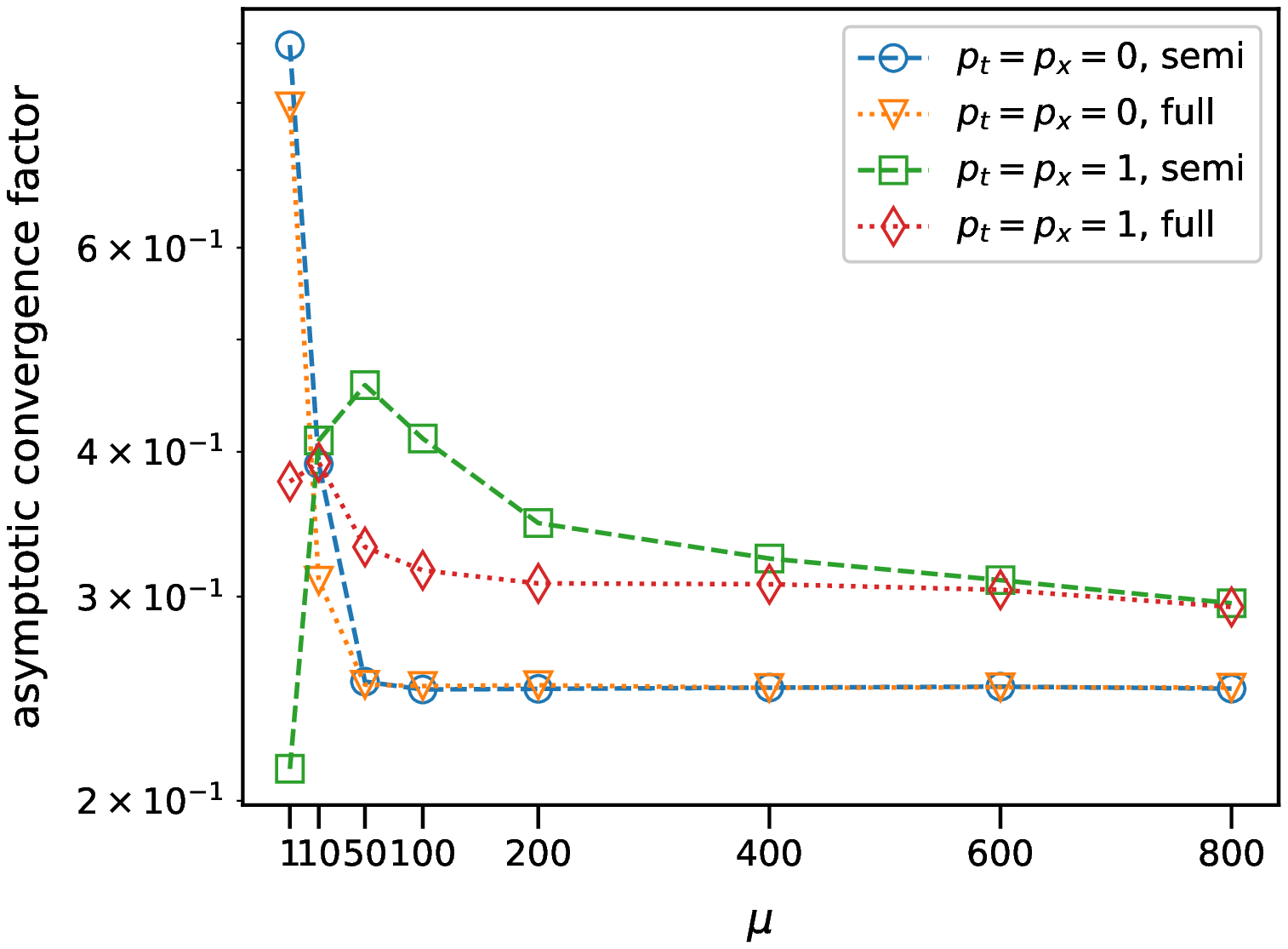}
\caption{Numerical convergence results: left for one spatial dimension, right for two spatial dimensions.}
\label{fig:Num}
\end{figure}

We now fix $CFL=600$ and vary the number of spatial elements to study the grid independence of the multigrid solver. The results can be seen in Figure~\ref{fig:Num1DNx}. For $p_t=p_x=1$ we can conclude a grid independence, while the convergence rate increases slightly when increasing the order of the DG approximation. 

\begin{figure}[h]
\centering
\includegraphics[width=0.5\textwidth]{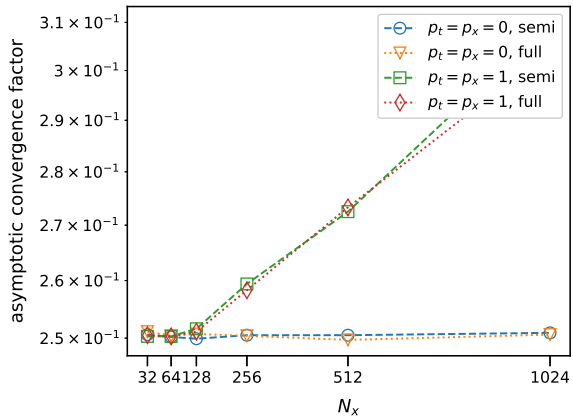}
\caption{Numerical convergence results two spatial dimension, $\mu=600$, $N=2^5$}
\label{fig:Num1DNx}
\end{figure}

\section{Conclusions}\label{SectionConclusions}
 
In this article we have applied the LFA to a space-time multigrid solver for the advection equation discretized with a space-time DG method. With the help of the analysis we calculated asymptotic convergence factors for the smoother and the two-grid method. The resulting Fourier symbols are complex since the spatial FV discretization with upwind flux results in a non-symmetric operator. For large CFL numbers we could analytically find promising asymptotic smoothing factors converging to $\frac{1}{\sqrt{2}}$ with increasing CFL number for both coarsening strategies independent of the temporal DG-SEM order. As for the smoother, it was difficult to find analytical expressions for the two-grid asymptotic convergence rates since they are based on the product of several complex Fourier symbols. We therefore calculated these numerically. The LFA gave excellent asymptotic convergence rates converging to $0.5$ for $p_t=0$ and decreasing to $0.375$ for $p_t=1$ for higher CFL numbers after some oscillations for small CFL numbers. The influence of the coarsening strategies on the convergence rates is minimal, with semi-coarsening in time resulting in slightly better asymptotic convergence rates for smaller CFL numbers.

For the numerical tests we considered non-periodic advection problems in one and two spatial dimensions with a space-time DG-SEM approximations and executed the numerical experiments in DUNE. We obtained asymptotic convergence rates of approximately $0.25$ for $p_t=p_x=0$ and $0.3$ for $p_t=p_x=1$ and high CFL numbers, independent of the coarsening strategy in the one-dimensional case. For two dimensions, asymptotic convergence rates of approximately $0.25$ were measured for high CFL numbers, independent of the  DG-SEM order and the coarsening strategy

The tests showed that the theoretical asymptotic convergence rates from the LFA were slightly larger than the convergence rates obtained in the numerical experiments. This can be explained by the different boundary conditions and more dimensions considered for in numerical experiments. Moreover, the coarsening strategy does not influence the results very much and simple block Jacobi smoothers can be used to get smoothing factors of $\frac{1}{\sqrt{2}}$.

However, solving the resulting space-time system at once results in large systems and it is thus advisable to either parallelize the solver or use a block multigrid solver for each space-time block.

\section*{Acknowledgements}

Gregor Gassner has been supported by the European Research Council (ERC) under the European Union’s Eights Framework
Program Horizon 2020 with the research project Extreme, ERC grant agreement no. 714487. 

\bibliographystyle{abbrv}
\bibliography{LFAVersbach}

\end{document}